\documentclass[12pt]{amsart}

\usepackage{amssymb,amsbsy}
\usepackage{latexsym,color,longtable}
\usepackage{verbatim,euscript}
\usepackage{fullpage,array}
\usepackage{tikz}
\usetikzlibrary{arrows,shapes,trees}

\numberwithin{equation}{section}
\usepackage[colorlinks=true,linkcolor=blue,urlcolor=violet,citecolor=magenta]{hyperref}

\input {cyracc.def}
\font\tencyr=wncyr10 
\font\tencyi=wncyi10 
\font\tencysc=wncysc10 
\def\rus{\tencyr\cyracc}
\def\rusi{\tencyi\cyracc}
\def\rusc{\tencysc\cyracc}

\newenvironment{proof*}
{\noindent {\sl Proof.}\quad }{\hfill    $\square$}

\catcode`\@=\active
\catcode`\@=11

\renewcommand{\@cite}[2]{[{{\bf #1}\if@tempswa , #2\fi}]}
\renewcommand{\@biblabel}[1]{[{\bf #1}]\hfill}
\catcode`\@=12

\newtheorem{thm}{Theorem}[section]
\newtheorem{lm}[thm]{Lemma}
\newtheorem{cl}[thm]{Corollary}
\newtheorem{prop}[thm]{Proposition}

\theoremstyle{remark}
\newtheorem{rmk}[thm]{Remark}

\theoremstyle{definition}
\newtheorem{ex}[thm]{Example}

\newcommand {\be}{{\mathfrak b}}

\newcommand {\g}{{\mathfrak g}}
\newcommand {\h}{{\mathfrak h}}
\newcommand {\ka}{{\mathfrak k}}

\newcommand {\te}{{\mathfrak t}}
\newcommand {\ut}{{\mathfrak u}}
\newcommand {\z}{{\mathfrak z}}

\newcommand {\gln}{{\mathfrak {gl}}_n}
\newcommand {\sln}{{\mathfrak {sl}}_n}

\newcommand {\sltn}{{\mathfrak {sl}}_{2n}}
\newcommand {\sltno}{{\mathfrak {sl}}_{2n+1}}

\newcommand {\slv}{{\mathfrak {sl}}(\BV)}

\newcommand {\spv}{{\mathfrak {sp}}(\BV)}
\newcommand {\spn}{{\mathfrak {sp}}_{2n}}

\newcommand {\sov}{{\mathfrak {so}}(\BV)}

\newcommand {\sone}{{\mathfrak {so}}_{2n}}
\newcommand {\son}{{\mathfrak {so}}_{n}}

\newcommand {\sfr}{\mathsf R}
\newcommand {\eus}{\EuScript}

\newcommand {\ap}{\alpha}

\newcommand {\lb}{\lambda}

\newcommand {\blb}{\boldsymbol{\lambda}}

\newcommand {\N}{{\mathcal N}}
\newcommand {\co}{{\mathcal O}}

\newcommand {\BV}{{\mathbb{V}}}

\newcommand {\BC}{{\mathbb C}}

\newcommand {\BZ}{{\mathbb Z}}

\newcommand {\ad}{{\mathrm{ad\,}}}

\newcommand {\hot}{{\mathsf{ht}}}

\newcommand {\Lie}{{\mathrm{Lie\,}}}
\newcommand {\Ker}{{\mathsf{Ker\,}}}
\newcommand {\Ima}{{\mathsf{Im\,}}}

\newcommand {\rk}{{\mathsf{rk}}}

\newcommand {\tri}{{\mathfrak{sl}}_2}

\newcommand {\GR}[2]{{\textrm{{\color{blue}\bf #1}}}_{#2}}
\newcommand {\GRt}[2]{{\color{blue}{\widetilde{\textrm{\bf #1}}}   }_{#2}   }

\newcommand {\ov}{\overline}
\newcommand {\un}{\underline}

\newcommand {\beq}{\begin{equation}}
\newcommand {\eeq}{\end{equation}}

\newcommand {\edva}{e^{\langle 2\rangle}}

\renewcommand{\le}{\leqslant}
\renewcommand{\ge}{\geqslant}
\renewcommand{\lg}{\langle}
\newcommand{\rg}{\rangle}

\newfam\Bbbfam%
\font\Bbbfont=msbm10 scaled 1200%
\font\olala=msam10 scaled 1200%
\font\Bbbsmallfont=msbm8%
\textfont\Bbbfam=\Bbbfont\scriptfont\Bbbfam=\Bbbsmallfont
\def\varnothing{\hbox {\Bbbfont\char'077}}
\def\square{\hbox {\olala\char"03}}

\begin{document}
\setlength{\parskip}{2pt plus 4pt minus 0pt}
\hfill {\scriptsize January 22, 2021} 
\vskip1ex

\title[Nilpotent orbits and mixed gradings]{Nilpotent orbits and mixed gradings of semisimple Lie algebras}
\author{Dmitri I. Panyushev}
\address{
Institute for Information Transmission Problems of the R.A.S., 
\hfil\break\indent  
Bolshoi Karetnyi per.~19, Moscow 127051, Russia}
\email{panyushev@iitp.ru}
\thanks{This research was funded by RFBR, project {\rus N0} 20-01-00515.}
\keywords{Centraliser, weighted Dynkin diagram, grading, involution}
\subjclass[2010]{17B08, 17B70, 14L30}
\begin{abstract}
Let $\sigma$ be an involution of a complex semisimple Lie algebra $\g$ and $\g=\g_0\oplus\g_1$ the 
related $\BZ_2$-grading. We study relations between nilpotent $G_0$-orbits in $\g_0$ 
and the respective $G$-orbits in $\g$. If $e\in\g_0$ is nilpotent and $\{e,h,f\}\subset\g_0$ is an
$\tri$-triple, then the semisimple element $h$ yields a $\BZ$-grading of $\g$. Our main tool is the 
combined $\BZ\times\BZ_2$-grading of $\g$, which is called a {mixed grading}. We prove, in particular, 
that if $e_\sigma$ is a regular nilpotent element of $\g_0$, then the weighted Dynkin diagram of 
$e_\sigma$, $\eus D(e_\sigma)$, has only isolated zeros. 
It is also shown that if $G{\cdot}e_\sigma\cap\g_1\ne\varnothing$, then the Satake diagram of $\sigma$
has only isolated black nodes and these black nodes occur among the zeros of 
$\eus D(e_\sigma)$. Using mixed gradings related to $e_\sigma$, we define an inner involution
$\check\sigma$ such that $\sigma$ and $\check\sigma$ commute. 
Here we prove that the Satake diagrams for both $\check\sigma$ and $\sigma\check\sigma$ have isolated black nodes.
\end{abstract}
\maketitle

\section{Introduction}
Let $G$ be a complex semisimple algebraic group with $\Lie G=\g$, $\mathsf{Inv}(\g)$ the set of 
involutions of $\g$, and $\N\subset\g$ the set of nilpotent elements. If $\sigma\in \mathsf{Inv}(\g)$, then $\g=\g_0\oplus\g_1$ is the corresponding 
$\BZ_2$-grading, i.e., $\g_i=\g_i^{(\sigma)}$ is the $(-1)^i$-eigenspace of $\sigma$. Let $G_0$ denote 
the connected subgroup of $G$ with $\Lie G_0=\g_0$. Study of the (nilpotent) $G_0$-orbits in $\g_1$
is closely related to the study of the (nilpotent) $G$-orbits in $\g$ meeting $\g_1$. The corresponding 
theory has been developed in~\cite{kr71}. In this article, we look at the $\BZ_2$-graded situation from 
another angle and study properties of nilpotent $G$-orbits meeting $\g_0$. Roughly speaking, the 
general problem is two-fold:

(a) \quad describe the $G$-orbits in $\N$ that meet $\g_0$ and then the $G_0$-orbits in $G{\cdot}e\cap\g_0$;

(b) \quad given an orbit $G_0{\cdot}e\subset\g_0\cap\N$, determine the (properties of) 
orbit $G{\cdot}e$.
\\
The well-known classification of the nilpotent orbits in $\son$ or $\spn$ via partitions comprise 
a solution to (a) for the pairs $(\g,\g_0)=(\sln,\son)$ or $(\sltn,\spn)$, i.e., for the
outer involutions of $\sln$. Here we provide some general results related to (b) and their applications.

For $e\in \g_0\cap\N$, an $\tri$-triple containing $e$, say $\{e,h,f\}$, can be chosen in $\g_0$. 
Combining the $\BZ_2$-grading $\g=\g_0\oplus\g_1$ and $\BZ$-grading $\g=\bigoplus_{i\in\BZ}\g(i)$ 
determined by $h$ yields a $\BZ\times\BZ_2$-grading $\g=\bigoplus_{(i,j)\in \BZ\times\BZ_2}\g_j(i)$, 
which is called a {\it mixed grading\/} related to $(\sigma,e)$ or just a $(\sigma,e)$-{\it grading}.
Here $e\in\g_0(2)$ and $h\in\g_0(0)$. Let $e_\sigma$ denote a regular nilpotent element of 
$\g_0=\g_{0}^{(\sigma)}$. We prove that the {weighted Dynkin diagram\/} of $G{\cdot}e_\sigma$, 
$\eus D(e_\sigma)$, has only isolated zeros. 
Moreover, if $G{\cdot}e_\sigma\cap\g_1^{(\sigma')}\ne\varnothing$ for some 
$\sigma'\in \mathsf{Inv}(\g)$, then the Satake diagram of $\sigma'$, denoted $\mathsf{Sat}(\sigma')$, 
has only isolated black nodes ({\sl IBN} for short), and these black nodes occur among the zeros of 
$\eus D(e_\sigma)$. Let $\mathsf c(\g)$ be the Coxeter number of a simple Lie algebra $\g$.
We show the orbit $G{\cdot}e_\sigma\subset\N$ is even whenever $\mathsf c(\g)$ is even.
It is also proved that if $\g_0$ is semisimple and $e$ is distinguished in $\g_0$, 
then $e$ remains even in $\g$ and the reductive part of the centraliser $\g^e$ is toral (i.e., $e$ is {\it 
almost distinguished\/} in $\g$), see Section~\ref{sect:mixed}.

For a simple Lie algebra $\g$, let $\kappa(\g)$ denote the {maximal} number of 
pairwise disjoint nodes in the Dynkin diagram. If $\mathsf c(\g)$ is even (i.e., 
$\g\ne\mathfrak{sl}_{2n+1}$), then we prove that there is a unique, up to $G$-conjugation, {\bf inner} 
involution $\vartheta$ such that $\eus D(e_\vartheta)$ has $\kappa(\g)$  isolated zeros. This $\vartheta$ is characterised by the property that $\g_1^{(\vartheta)}$ contains a regular nilpotent element of $\g$.
Let $B$ be a Borel subgroup of $G$ and $U=(B,B)$. Using $\vartheta$, we also show that $B$ has a dense 
orbit in $\ut'=[\ut,\ut]$, where $\ut=\Lie U$, and  if $\co$ is the dense $B$-orbit in $\ut'$, then 
$G{\cdot}\co=G{\cdot}e_\vartheta$, see Section~\ref{sect:max_int_inv}. 

As in~\cite{divis}, we say that the orbit $G{\cdot}e\subset \N$ is {\it divisible}, if $\frac{1}{2}\eus D(e)$ 
is again a weighted Dynkin diagram. Then the respective nilpotent orbit is denoted by 
$G{\cdot}\edva$. 
There is an interesting link between mixed gradings and divisible $G$-orbits. If a $(\sigma,e)$-grading 
has the property that $\dim\g_0(0)=\dim\g_1(4)$, then we prove that $G{\cdot}e$ is divisible, $\g_0$ 
is semisimple, and  $\dim\g_0(4k+2)=\dim\g_1(4k+2)$ for all $k\in\BZ$. Moreover, both $e$ and $\edva$ 
are almost distinguished in $\g$ and all such instances are classified, see Section~\ref{sect:divis} for 
details. 

Given $\sigma\in\mathsf{Inv}(\g)$, we define $\Upsilon(\sigma)\in\mathsf{Inv}(\g)$, if $\sigma$ 
is not an inner involution of $\mathfrak{sl}_{2n+1}$. Constructing the map $\Upsilon$ depends on a 
choice of $e_\sigma\in \g_0^{(\sigma)}\cap\N$ and exploits the mixed grading related to $(\sigma, e_\sigma)$, 
see Section~\ref{sect:new-invol}. Here $\check\sigma:=\Upsilon(\sigma)$ is always inner, and the involutions
$\sigma$ and $\check\sigma$ commute. Hence $\sigma\check\sigma\in\mathsf{Inv}(\g)$ 
and $\sigma$, $\sigma\check\sigma$ belong to the same connected component of the group
$\mathsf{Aut}(\g)$. The map $\Upsilon$ has the property that $e_\sigma\in \g_1^{(\check\sigma)}$ and 
$e_\sigma\in \g_1^{(\sigma\check\sigma)}$. This implies that the Satake diagrams
$\mathsf{Sat}(\check\sigma)$ and $\mathsf{Sat}(\sigma\check\sigma)$ have only {\sl IBN}.\/ We describe 
a method that allows us to compute (the conjugacy class of) $\check\sigma$ or $\sigma\check\sigma$.
We also
discuss the property that $\sigma$ and $\sigma\check\sigma$ are $G$-conjugate, and its connection
to the divisibility of $G{\cdot}e_\sigma$. The commuting involutions $\sigma$ and $\check\sigma$
provide a $\BZ_2\times\BZ_2$-grading of $\g$, and our construction based on a mixed grading yields
an explicit model for it.

\un{Main notation}. Let $\h$ be a fixed Cartan subalgebra of $\g$, $\Delta$ the root system of $(\g,\h)$, 
and $\Pi$ a set of simple roots. If $\gamma\in\Delta$, then $\g_\gamma$ is the root space in $\g$. 
Write $\g^x$ or $\z_\g(x)$ for the centraliser of $x\in\g$ in $\g$. More generally, 
$\z_\g(M)=\cap_{x\in M}\z_\g(x)$ for a subset $M$ of $\g$. If $G{\cdot}e\subset\N$ is a nonzero  
orbit, then $\eus D(e)$ is its weighted Dynkin diagram. A direct sum of Lie algebras is denoted by
`$\dotplus$'.
\\ \indent
Our man reference for algebraic groups and Lie algebras is~\cite{t41}.
We refer to~\cite{CM} for generalities on nilpotent elements (orbits) and their centralisers.

\section{Preliminaries on nilpotent orbits and involutions} 
\label{sect:prelim}

\noindent
Suppose for a while that $\g$ is a reductive algebraic Lie algebra. We say that $x\in \g$ is {\it regular}, if 
$\dim\g^x=\rk(\g)$. The set of regular elements is denoted by $\g_{\sf reg}$.
Let $\N$ be the set of nilpotent elements of $\g$. For $e\in\N\setminus\{0\}$, let $\{e,h,f\}$ be an 
$\tri$-triple in $\g$. That is, $[h,e]=2e$, $[e,f]=h$, and $[h,f]=-2f$. The semisimple element $h$ is called 
a {\it characteristic\/} of $e$. W.l.o.g. one may assume that $h\in\h$ and $\ap(h)\ge 0$ for all $\ap\in\Pi$. 
By a celebrated result of Dynkin, one then has $\ap(h)\in\{0,1,2\}$~\cite[Theorem\,8.3]{Dy52}. The {\it weighted Dynkin diagram\/} of $G{\cdot}e$ (=\,of $e$), $\eus D(e)$,  is 
the Dynkin diagram of $\g$ equipped with labels $\{\ap(h)\}_{\ap\in\Pi}$. The {\it set of zeros\/} of 
$\eus D(e)$ is the subset $\Pi_0=\{\ap\in\Pi\mid \ap(h)=0\}$.

Let $\g=\bigoplus_{i\in\BZ}\g(i)$ be the $\BZ$-grading determined by $h$, i.e.,  
$\g(i)=\{v\in \g \mid [h,v]=iv\}$. Then $\g^e$ inherits this $\BZ$-grading and $\g(0)=\g^h$. 
Recall some standard definitions related to nilpotent elements and $\tri$-triples.  
A nonzero $e\in\N$ is said to be

\textbullet\quad  {\it even}, if the $h$-eigenvalues in $\g$ are even;
\\ \indent
\textbullet\quad  {\it distinguished}, if $\z_\g(e,h,f)$ is the centre of $\g$ (i.e.,
$\z_\g(e,h,f)=0$, if $\g$ is semisimple); 
\\ \indent 
\textbullet\quad  {\it almost distinguished}, if $\z_\g(e,h,f)$ is a toral Lie algebra (=\,Lie algebra of a torus);
\\ 
Write $\te_n$ for an $n$-dimensional toral Lie algebra. 
For $e\in\g_{\sf reg}\cap\N$, any $\tri$-triple $\{e,h,f\}$ is said to be {\it principal} (in $\g$).
Set $\g({\ge}j)=\bigoplus_{i\ge j}\g(i)$.  It is well known that 
\begin{itemize}  
\item[{\sf (i)}] \  $\ad e:\g(i)\to \g(i+2)$ is injective (resp. surjective) if $i\le -1$ (resp. $i\ge -1$). Hence
 $\g^e\subset \g({\ge} 0)$, $\dim\g^e=\dim\g(0)+\dim\g(1)=\dim\g^h+\dim\g(1)$, and $e$ is even if and only if $\dim\g^e=\dim\g^h$. Furthermore, $(\ad e)^i: \g(-i) \to \g(i)$ is bijective.
\item[\sf (ii)] \ a regular element is distinguished and a distinguished element is even~\cite[Theorem\,8.2.3]{CM}; however, there exist 
almost distinguished non-even elements.  \label{n-i}
\item[{\sf (iii)}] \ $\z_\g(e,h,f)=\g^e(0)$ is a maximal reductive subalgebra of $\g^e$~\cite[3.7]{CM}. We also 
write $\g^e_{\sf red}$ for it. Then $\dim\g^e_{\sf red}=\dim\g(0)-\dim\g(2)$ and $e$ is distinguished 
(resp. almost distinguished) if and only if $\g^e$ has no non-central semisimple elements (resp. is solvable).
\item[{\sf (iv)}] \ The nilradical of $\g^e$, $\g^e_{\sf nil}$, is contained in $\g({\ge} 1)$ and
$\dim\g^e_{\sf nil}=\dim\g(1)+\dim\g(2)$. 
\end{itemize}

\noindent From now on, $\g$ is assumed to be semisimple. 
Let $\sigma$ be an involution of $\g$ and $\g=\g_0\oplus\g_1$ the corresponding $\BZ_2$-grading.
Then $\g_0=\g^\sigma$ is reductive, but not necessarily semisimple. We also say that  $(\g,\g_0)$ is a 
{\it symmetric pair}. Whenever we wish to stress that $\g_i$ is defined via certain $\sigma\in\mathsf{Inv}(\g)$, especially when several involutions are being considered simultaneously, we write $\g_i^{(\sigma)}$ for it. We can also write $\g^\sigma$ in place of $\g^{(\sigma)}_0$.
\\ \indent
The centraliser  $\g^x$ is $\sigma$-stable for any $x\in \g_i$, hence $\g^x=\g^x_0\oplus\g^x_1$. It is 
known that
\\ \indent \textbullet \ \ if $x\in\g_1$, then $\dim G_0{\cdot}x=\frac{1}{2}\dim G{\cdot}x$; that is, 
$\dim\g_1-\dim\g_1^x=\dim\g_0-\dim\g_0^x$~\cite[Proposition\,5]{kr71}. 
\\ \indent \textbullet \ \ if $0\ne x\in \g_0$, then $\dim\g_0^x+\rk(\g)> \dim\g_1^x$~\cite[Theorem\,4.4]{ANT13}.
\\[.7ex]
It is well known that the real forms of $\g$ are represented by their {\it Satake diagrams} (see 
e.g.~\cite[Ch.\,4,\S\,4.3]{t41}) and there is a one-to-one correspondence between the real forms and 
$\BZ_2$-gradings of $\g$. Thereby, one associates the Satake diagram to an involution (symmetric 
pair), cf.~\cite{spr87}. A Satake diagram of $\sigma$, $\mathsf{Sat}(\sigma)$, is the Dynkin diagram of $\g$, with black 
and white nodes, where certain pairs of white nodes can be joined by an arrow. Let $x\in\g_1$ be a 
generic semisimple element. Then $\g^x_1$ is a toral subalgebra (a {\it Cartan subspace}\/
of $\g_1$) and $\mathsf{Sat}(\sigma)$ encodes the structure of $\g^x_0$. In particular, the subdiagram of
black nodes in $\mathsf{Sat}(\sigma)$ represents the Dynkin diagram of $[\g^x,\g^x]=[\g^x_0,\g^x_0]$, 
while the number of arrows equals $\dim(\g^x_0/[\g^x_0,\g^x_0])$. 
Some features of Satake diagrams in the setting of $\BZ_2$-gradings are discussed 
in~\cite[Sect.\,2]{BSM14}.  

\begin{ex}   \label{ex:max-rang}
As $G/G_0$ is a spherical homogeneous space,  $\dim\g_1=\dim (G/G_0)\le \dim B$. Hence
$\dim\g_0\ge \dim U$ and  
$\dim\g_1-\dim\g_0\le\rk(\g)$ for any $\sigma\in\mathsf{Inv}(\g)$. If $\dim\g_1-\dim\g_0=\rk(\g)$,
then $\sigma$ is said to be of {\it maximal rank}. (In~\cite{spr87}, such involutions are called {\it split}.)
Equivalently, $\g_1$ contains a Cartan subalgebra of $\g$. For any simple $\g$, there is a unique, up to 
$G$-conjugacy, involution of maximal rank, and we denote it by $\vartheta_{\sf max}$. In this case, 
$\g^x\cap \g_0^{(\vartheta_{\sf max})} =\{0\}$ for a generic $x\in\g_1^{(\vartheta_{\sf max})}$. Hence 
$\mathsf{Sat}(\vartheta_{\sf max})$ has neither black nodes nor arrows. Yet another characterisation is 
that $\vartheta_{\sf max}$
corresponds to a split real form of $\g$, see~\cite[Ch.\,4,\S\,4.4]{t41}.
\end{ex}

\begin{rmk}    \label{rem:Levon}
By a fundamental result of Antonyan, for any $\sigma\in\mathsf{Inv}(\g)$ and an $\tri$-triple 
$\{e,h,f\}\subset\g$, one has $G{\cdot}e\cap\g_1\ne\varnothing$ if and only if 
$G{\cdot}h\cap\g_1\ne\varnothing$, see~\cite[Theorem\,1]{an}. This readily implies that 
{$\sigma=\vartheta_{\sf max}$ if and only if $G{\cdot}x\cap\g_1\ne\varnothing$ for {\bf any} 
$x\in\g$ \ (\cite[Theorem\,2]{an}).} For arbitrary $\sigma$ and $e\in\N$, this means that 
$G{\cdot}e\cap\g_1\ne\varnothing$ if and only if {\sf (i)} the set of black nodes of $\mathsf{Sat}(\sigma)$ 
is contained in the set of zeros of $\eus D(e)$ and {\sf (ii)} $\ap(h)=\beta(h)$ whenever the nodes 
$\ap,\beta\in\Pi$ are joined by an arrow in $\mathsf{Sat}(\sigma)$.
\end{rmk}

\section{Mixed gradings of semisimple Lie algebras} 
\label{sect:mixed}

\noindent
A {\it mixed grading\/} of $\g$ is a grading via the group $\BZ\times \BZ_2$. We consider only mixed 
gradings of a special form. Given $\sigma\in \mathsf{Inv}(\g)$ and a nonzero $e\in\g_0\cap\N$, take an 
$\tri$-triple $\{e,h,f\}$ in $\g_0$. Then we set $\g_j(i)=\{v\in \g_j \mid [h,v]=iv\}$ and consider the mixed 
grading
\beq   \label{eq:mixed}
    \g=\bigoplus_{j\in\BZ_2}\bigoplus_{i\in\BZ}\g_j(i)=\bigoplus_{(i,j)\in\BZ\times \BZ_2}\g_j(i) . 
\eeq
We say that~\eqref{eq:mixed} is a {\it mixed grading related to\/} $(\sigma, e)$ or just
a $(\sigma,e)$-{\it grading}. Here $e\in\g_0(2)$ and $f\in\g_0(-2)$.
For such a mixed grading, it follows from {\sf (i)} in Section~\ref{sect:prelim} that 
\beq            \label{eq:mix-inj-surj}
\left\{  \begin{array}{l}
   \ad e: \g_j(i)\to \g_j(i+2) \text{ is surjective for $j=0,1$ and $i\ge -1$}. \\
   \ad e: \g_j(i)\to \g_j(i+2) \text{ is injective for $j=0,1$ and $i\le -1$}. 
\end{array} \right .
\eeq
Letting $d_j(i)=\dim\g_j(i)$, we have $d_j(i)\ge d_j(i+2)$ for $i\ge -1$ and $d_j(i)=d_j(-i)$.
Below, we mostly consider {\bf even} nilpotent elements of $\g_0$.
However,  if $e\in \g_0$ is even in $\g_0$, then $e$ is not necessarily even in $\g$.

\begin{lm}   \label{lm:chet-nechet}
If\/ $\g$ is simple and $e$ is even in $\g_0$, then the $h$-eigenvalues in $\g_1$ are either all even or 
all odd.
\end{lm}
\begin{proof}
Write $\g_1=\g_1^+\oplus \g_1^-$, where the $h$-eigenvalues in $\g_1^+$ (resp. $\g_1^-$) are even 
(resp. odd). Since the $h$-eigenvalues in $\g_0$ are even, we have $[\g_0, \g_1^\pm]\subset \g_1^\pm$; 
also $[\g_1^+,\g_1^-]\subset \g_0$. On the other hand, the $h$-eigenvalues in $[\g_1^+,\g_1^-]$ are odd. 
This implies that $[\g_1^+,\g_1^-]=0$. Therefore, letting $\tilde\g_0=\g_0\oplus \g_1^+$ and 
$\tilde\g_1=\g_1^-$, we obtain another $\BZ_2$-grading of $\g$. As is well known, for a symmetric pair
$(\g,\g_0)$ with simple Lie algebra $\g$, $\g_0$ is a maximal proper reductive subalgebra.
Consequently, either $\tilde\g_0=\g_0$ or $\tilde\g_0=\g$.
\end{proof}

\begin{lm}         \label{lm:pochti-otm}
Suppose that $\g_0$ is semisimple and $e\in\g_0\cap\N$ is distinguished in $\g_0$. Then 
$e$ is almost distinguished and even in $\g$.
\end{lm}
\begin{proof}
The assumptions imply that $\z_\g(e,h,f)\cap \g_0=\{0\}$. Therefore,  
$\z_\g(e,h,f)\subset \g_1$. Then $\z_\g(e,h,f)$ is reductive and abelian, hence toral. 
\\ \indent
Since $e$ is even in $\g_0$, it follows from Lemma~\ref{lm:chet-nechet} that either $\g_1=\g_1^+$ or 
$\g_1=\g_1^-$.  Assume that $\g_1=\g_1^-$.  Since $(\g_1^-)^h=\{0\}$, we obtain
$\z_\g(e,h,f)=\z_{\g_0}(e,h,f)=0$. Therefore, $e$ is distinguished in $\g$ and hence even. This contradicts the fact that $\g_1^-$ is nontrivial. Thus, $\g_1=\g_1^+$ and  
$e$ is even in $\g$, although not necessarily distinguished.
\end{proof}

For $\g\in\{\slv,\sov,\spv\}$, the nilpotent orbits are represented by partitions of $\dim \BV$, 
see~\cite[Ch.\,5]{CM}. There are also simple algorithms for obtaining $\eus D(e)$
via $\blb=\blb(e)=\blb(G{\cdot}e)$, which go back to Springer and Steinberg~\cite[IV.4]{ss}, and 
here $G{\cdot}e\subset\g$ is even if and only if all parts of $\blb(e)$ have the same parity. 
\begin{ex}   
\label{ex:not-even}
a) \ If $\sigma\in\mathsf{Inv}(\mathfrak{sl}_{2n+1})$ is inner, then 
$\g_0=\mathfrak{sl}_k\dotplus\mathfrak{sl}_{2n+1-k}\dotplus\te_1$ with $k<2n+1-k$. If 
$e\in\g_{0,\sf reg}$, 
then $\blb(e)=(2n+1-k,k)$ and $e$ is {\bf not} even in $\g$. Here $\g_0$ is not semisimple and 
$\g_1=\g_1^-$. 

b) \ Another example is $\g=\spn$ and $\g_0=\mathfrak{sp}_{2k}\dotplus \mathfrak{sp}_{2n-2k}$ with
$0<k<n$.
If $e$ is regular in $\mathfrak{sp}_{2k}$, then  
$\blb(G{\cdot}e)=(2k,1,\dots,1)$, which means that $e$ is not even in $\g$.

c) \ It can happen that $\g_0$ is simple and $e\in \g_0$ is even, but $e$ is not even in $\g$. 
For instance, take $\sigma\in \mathsf{Inv}(\GR{F}{4})$ such that $(\GR{F}{4})^\sigma=\mathfrak{so}_9$.
Let $e\in \mathfrak{so}_9$ be such that $\blb(e)=(3,3,3)$.
Then $\eus D(e)$ is \ \raisebox{-.5ex}{\begin{tikzpicture}[scale= .6, transform shape]
\tikzstyle{every node}=[circle, draw, fill=orange!30]
\node (a) at (0,0) {\bf 1};
\node (b) at (1,0) {\bf 0};
\tikzstyle{every node}=[circle, draw, fill=white!55]
\node (c) at (2.5,0) {\bf 0};
\node (d) at (3.5,0) {\bf 2};
\foreach \from/\to in {a/b,  c/d}  \draw[-] (\from) -- (\to);
\draw (1.4, .07) -- +(.7,0);
\draw (1.4, -.07) -- +(.7,0);
\end{tikzpicture}}.
Here (and in Tables~\ref{table:odin},\,\ref{table:dva} below) the shaded nodes in the weighted Dynkin diagrams represent the {\bf short} simple 
roots. 
\end{ex}
Let $G(0)$ (resp. $G_0(0)$) be the connected subgroup of $G$ with Lie algebra $\g(0)$
(resp $\g_0(0)$).

\begin{lm}    \label{lm:finite-orb}
For any mixed grading of $\g$,  $G_0(0)$ has finitely many orbits in any $\g_j(i)$ if $i\ne 0$. 
Consequently, for $i\ne 0$, there is a dense $G_0(0)$-orbit in $\g_j(i)$ and hence $d_j(i)\le d_0(0)$.
\end{lm}
\begin{proof}
The presence of the grading shows that  $[\g,x] \cap \g_j(i)=[\g_0(0),x]$ for any $x\in \dim\g_j(i)$.
By Vinberg's lemma~\cite[\S\,2]{vi76}, this implies that the intersection of any $G$-orbit with
$\g_j(i)$ consists of finitely many $G_0(0)$-orbits.  For $i\ne 0$,
all elements of $\g_j(i)$ are nilpotent. Therefore,
there are finitely many (nilpotent) $G$-orbits meeting $\g_j(i)$ with $i\ne 0$.  
\end{proof}

\begin{thm}  \label{prop:reg-in-g0}
Suppose that $\g$ is simple, $e\in\g_{0,{\sf reg}}\cap\N$, and $\{e,h,f\}\subset \g_0$ is a 
($\g_0$-principal) $\tri$-triple. Then
\begin{enumerate}
\item  \ \label{item-1} $[\g^h,\g^h]\simeq \tri\dotplus\dots\dotplus \tri=(\tri)^{k}$ for some $k\ge 0$
and $\eus D(e)$ has only \emph{isolated}\/ zeros. More precisely, here
$\dim\g^h= \rk(\g)+2k$ and $\eus D(e)$ contains exactly $k$ isolated zeros; 
\item \ \label{item-2} $\dim\g^e_{\sf nil}\le 2\rk(\g_0)$, and if the equality holds, then $\g_0$ is semisimple.
\end{enumerate}
\end{thm}
\begin{proof}  Consider a mixed grading related to $(\sigma,e)$.
\\ \indent
(1) The centraliser of $h$ in $\g$ is reductive and $\BZ_2$-graded:
$\g^h=\g_0(0)\oplus \g_1(0)$. Since $e$ is regular nilpotent in $\g_0$, $\g_0(0)$ is a Cartan subalgebra of 
$\g_0$, and therefore $[\g^h,\g^h]_0$ is commutative. For a semisimple Lie algebra $[\g^h,\g^h]$, this is 
only possible if all its simple factors are isomorphic to $\tri$. Each simple factor $\tri$ of $\g^h$ 
corresponds to a simple root, and the corresponding set of $k$ simple roots gives rise to a totally 
disconnected subset of the Dynkin diagram.  Recall that the set of zeros of $\eus D(e)$
represent the Dynkin diagram of $[\g^h,\g^h]$.  The rest is clear.

(2) Recall that $\dim\g^e_{\sf nil}=\dim\g(1)+\dim\g(2)$. By Lemma~\ref{lm:chet-nechet},
either $\g_1=\g_1^+$ or $\g_1=\g_1^-$.
\\
If $\g_1=\g_1^+$, then
$\g(1)=0$, and $\dim\g^e_{\sf nil}=\dim\g_0(2)+\dim\g_1(2)\le 2\dim\g_0(0)=2\rk(\g_0)$.
\\
If $\g_1=\g_1^-$, then $\g(1)=\g_1(1)$ and $\g(2)=\g_0(2)$, with the similar estimate.

In both cases, if $\dim\g^e_{\sf nil}= 2\rk(\g_0)$, then $\dim\g_0(2)=\dim\g_0(0)$. Since $\{e,h,f\}\subset [\g_0,\g_0]$
and $\g_0(i)\subset [\g_0,\g_0]$ for $i\ne 0$, we have
\[
   \dim \g_0(2)=\dim [\g_0,\g_0](2)\le \dim[\g_0,\g_0](0)\le \dim \g_0(0) .
\]
Hence $[\g_0,\g_0](0)= \g_0(0)$ and therefore $[\g_0,\g_0]=\g_0$, i.e., $\g_0$ is semisimple.
\end{proof}

{\bf Remark.} It is proved by Broer~\cite{br94} that if $e$ is even, $\eus D(e)$ has isolated zeros, and 
the set of zeros $\Pi_0$ consists of short roots, then the closure of ${G{\cdot}e}$ is normal. (If $\g$ 
is simply-laced, then all roots are assumed to be short.)

In Theorem~\ref{prop:reg-in-g0}, $e\in\g_{0,\sf reg}$ is not necessarily even in $\g$, and we 
characterise below possible 
exceptions. Write $\mathsf c(\g)$ for the {\it Coxeter number\/} of $\g$.  Let $\Delta^+$ be the set of 
positive roots corresponding to $\Pi$ and $\theta\in\Delta^+$ the highest root. For $\gamma\in\Delta$ 
and $\ap\in\Pi$, let $[\gamma:\ap]$ be the coefficient of $\ap$ in the expression of $\gamma$ via $\Pi$. 
Then $\hot(\gamma):=\sum_{\ap\in\Pi}[\gamma:\ap]$ and $\mathsf c(\g)=\hot(\theta)+1$. 

\begin{prop}   \label{thm:cox-even}
Suppose that $\g$ is simple and $e\in\g_{0,{\sf reg}}\cap\N$. If\/ $\mathsf c(\g)$ is even, then 
$e$ is even in $\g$. Moreover, $e$ is not even in $\g$ if and only if\/ $\g=\mathfrak{sl}_{2n+1}$ and $\sigma$ is inner.
\end{prop}
\begin{proof} 
{\sf (1)} \  If $\g_0$ is semisimple, then $e$ is even by Lemma~\ref{lm:pochti-otm}.
\\ \indent
{\sf (2)} \ If $\g_0$ is not semisimple,  then $\sigma$ is inner and $\g_0$ is a Levi subalgebra of a 
(maximal) parabolic subalgebra with abelian nilradical. Namely, there is $\beta\in\Pi$ such that 
${[\theta:\beta]=1}$ and the set of simple roots of $\g_0$ is $\Pi_0:=\Pi\setminus\{\beta\}$. Set 
$\Delta_\beta(i)=\{\gamma\in \Delta\mid [\gamma:\beta]=i\}$. Then 
$\Delta=\Delta_\beta(-1)\cup\Delta_\beta(0)\cup\Delta_\beta(1)$,
\[
\g_0=\h\oplus (\bigoplus_{\gamma\in\Delta_\beta(0)}\g_\gamma) \ \text{ and } \ 
\g_1=\bigoplus_{\gamma\in\Delta_\beta(-1)\cup\Delta_\beta(1)}\g_\gamma =\g(-1)\oplus\g(1) .
\]
Here $\g(1)$ is a simple $\g_0$-module, with the highest (resp. lowest) weight $\theta$ (resp. $\beta)$
w.r.t. $\Delta_\beta(0)^+=\Delta_\beta(0)\cap\Delta^+$. If $\theta=\beta+\sum_{\ap_i\in\Pi_0}n_i\ap_i$, 
then $\mathsf c(\g)=\sum n_i+2$. Let $\{e,h,f\}$ be a principal $\tri$-triple in $[\g_0,\g_0]\subset\g_0$ 
such that $h\in\h$ and $e=\sum_{\ap_i\in\Pi_0}e_{\ap_i}$. Then $\ap_i(h)=2$ for all $\ap_i\in\Pi_0$ and 
$\theta(h)=-\beta(h)$. It follows that $\theta(h)=\sum n_i=\mathsf c(\g)-2$. Thus, the eigenvalue 
$\theta(h)$ is even if and only if $\mathsf c(\g)$ is even. In this case all $h$-eigenvalues in $\g_1$ are 
even (Lemma~\ref{lm:chet-nechet}) and hence $e$ is even in $\g$. It remains to observe that
$\mathsf c(\g)$ is odd if and only if $\g=\mathfrak{sl}_{2n+1}$.
\end{proof}
Let $\kappa(\g)$ denote the maximal number of pairwise disjoint nodes in the Dynkin diagram of $\g$. 
By Theorem~\ref{prop:reg-in-g0}, if $\{e_\sigma,h_\sigma,f_\sigma\}$ is a principal $\tri$-triple in $\g_0=\g^\sigma$ 
for some $\sigma\in\mathsf{Inv}(\g)$, then $\dim\g^{h_\sigma} \le \rk(\g)+2\kappa(\g)$. We prove in 
Section~\ref{sect:max_int_inv} that, for $\g\ne \mathfrak{sl}_{2n+1}$, there is always an inner 
involution $\vartheta$ such that $\dim\g^{h_\vartheta} = \rk(\g)+2\kappa(\g)$  
and hence  $\eus D(e_\vartheta)$ has the maximal possible number of isolated zeros.

\begin{rmk}   \label{rmk:max-zeros}
It is readily seen that if $\g\ne \GR{D}{2n}$, then $\kappa(\g)=\left[\frac{\rk(\g)+1}{2}\right]$, while
$\kappa(\GR{D}{2n})=n+1$. A uniform but more fancy expression that can also be verified 
case-by-case is
\beq  \label{eq:kappa}
    \kappa(\g)= \#\{\gamma\in\Delta^+\mid \hot(\gamma)=
    \left[ (\mathsf{c}(\g)+1)/2 \right]=:a \}.     
\eeq
Note that $\bigoplus_{\gamma: \hot(\gamma)\ge a} \g_\gamma$ is an abelian ideal of the Borel 
subalgebra $\be=\h\oplus (\bigoplus_{\gamma\in\Delta^+}\g_\gamma)$. Using a result of Sommers 
related to the theory of {\sl ad}-nilpotent ideals of $\be$~\cite[Theorem\,6.4]{som05}, one obtains 
inequality  ``$\ge$'' in
\eqref{eq:kappa}. Moreover, if $\mathsf c(\g)$ is even, which only excludes  
$\g=\mathfrak{sl}_{2n+1}$, then I can give a case-free proof of~\eqref{eq:kappa}. 
\end{rmk}

Assume that $e\in\g_0=\g^\sigma$ is even in $\g$,  and let 
$\g=\bigoplus_{i,j}\g_j(2i)$ be a $(\sigma,e)$-grading. Recall that 
$d_j(i)=\dim\g_j(i)$,  $d_0(0)\ge d_j(i)$ for $i\ne 0$ (Lemma~\ref{lm:finite-orb}), 
and $d_j(i)\ge d_j(i+2)$ for $i\ge 0$, cf. Eq.~\eqref{eq:mix-inj-surj}. Consider the even integers
$m_j=\max\{k\mid d_j(k)\ne 0\}$ for $j=0,1$.

\begin{prop}    \label{prop:0=2}
Given $\sigma\in {\sf Inv}(\g)$ and a $(\sigma,e)$-grading of\/ $\g$, suppose that $d_0(0)=d_1(2)$. Then 
\begin{enumerate}
\item $G(0){\cdot}e\cap \g_1(2)\ne \varnothing$. In particular, $G{\cdot}e\cap\g_1\ne\varnothing$;
\item $e$ is almost distinguished in $\g$;
\item  $|m_0-m_1|\le 2$ and $d_0(0)\le d_1(0)$.
\end{enumerate}
\end{prop}
\begin{proof}
 (1) \ For $e'\in\g_1(2)$, the space $[\g_1(-2),e']$ is the orthogonal complement of $\g^{e'}_0(0)$ in
$\g_0(0)$ w.r.t{.}~the Killing form. Let $\co$ be the dense $G_0(0)$-orbit in $\g_1(2)$. If $e'\in\co$, then $\g^{e'}_0(0)=\{0\}$
for the dimension reason. Hence $[\g_1(-2),e']=\g_0(0)\ni h$. That is, there is $f'\in \g_1(-2)$ such that 
$\{e',h,f'\}$ is an $\tri$-triple. Because $G(0)$ is the centraliser of $h$ in $G$, this also implies that
$e'\in G(0){\cdot}e$, see \cite[Theorem\,1(4)]{vi79}.

(2) Since both $\ad e': \g_0(0)\to \g_1(2)$ and $\ad e': \g_1(0)\to \g_0(2)$ are onto and $d_0(0)=d_1(2)$,
we see that $\g(0)^{e'}=\g^{e'}_{\sf red}\in \g_1(0)$. Hence $\g^{e'}_{\sf red}$ is a toral subalgebra, and 
so is $\g^{e}_{\sf red}$.

(3) Using the $\tri$-triple $\{e',h,f'\}$ with $e'\in\g_1(2)$, we see that 
$\ad e'$ takes $\g_j(i)$ to $\g_{j{+}1}(i{+}2)$ and $\ad e':\g({\ge} 0)\to \g({\ge} 2)$ is 
onto. 
Since $\g_0(m_0)$ and $\g_1(m_1)$ are in the range of $\ad e'$, one has 
$\g_0(m_1-2)\ne 0$ and $\g_1(m_0-2)\ne 0$, i.e., $|m_0-m_1|\le 2$. 
Note also that  $d_1(0)\ge d_1(2)=d_0(0)$.
\end{proof}

\noindent
The hypothesis of Proposition~\ref{prop:0=2} is rather restrictive. It means that the {\bf total} number 
of $\langle e,h,f\rangle$-modules in $\g_0$ equals the number of {\bf nontrivial} 
$\langle e,h,f\rangle$-modules in $\g_1$, i.e., roughly speaking, $\g_0$ cannot be much bigger than 
$\g_1$. Actually, assertions of Proposition~\ref{prop:0=2} fail, if $\g_0$ is considerably larger than $\g_1$. 
For instance, let $e\in\g_0$ be regular in $\g_0$ for $(\g,\g_0)=(\sone, \mathfrak{so}_{2n-1})$ with 
$n\ge 3$. Then $d_0(0)=n-1$, $d_1(0)=d_1(2)=1$,  $m_0=4n-6$, and $m_1=2n-2$. In this case, we 
also have $G{\cdot}e\cap\g_1=\varnothing$. The same conclusions hold for 
$(\mathfrak{so}_{2n+1},\mathfrak{so}_{2n})$ as well.

For $e\in\g_0$, it can happen that $G{\cdot}e\cap\g_1\ne\varnothing$, while 
$G(0){\cdot}e\cap\g_1(2)=\varnothing$. By the construction, we have $e\in\g_0(2)\subset\g(2)$ and 
$G(0){\cdot}e$ is the dense orbit in $\g(2)$. But this dense $G(0)$-orbit does not necessarily meet 
$\g_1(2)$. A simple possible reason for that is that $\g_0(2)$ contains a $1$-dimensional $G(0)$-module.

\begin{prop}    \label{prop:IBN}
Given $\sigma\in {\sf Inv}(\g)$ and $e\in \g_{0,{\sf reg}}\cap\N$, suppose that 
$G{\cdot}e\cap\g_1^{(\sigma')}\ne\varnothing$ for some $\sigma'\in\mathsf{Inv}(\g)$. Then the Satake 
diagram $\mathsf{Sat}(\sigma')$ has only \emph{isolated black nodes} ({\sl IBN} for short). Moreover, the 
set of black nodes of\/ $\mathsf{Sat}(\sigma')$ is contained in the set zeros of\/ $\eus D(e)$.
\end{prop}
\begin{proof}  Set $\tilde\g_i=\g_{i}^{(\sigma')}$.
If $x\in\tilde\g_{1}$ is a generic semisimple element, then $\mathsf{Sat}(\sigma')$ has
{\sl IBN}\/ if and only if $[\g^x,\g^x]\simeq (\tri)^{k'}$ for some $k'$
(and then $\mathsf{Sat}(\sigma')$ has exactly $k'$ isolated black nodes). 

Take any $e'\in G{\cdot}e\cap\tilde\g_1$. There is an $\tri$-triple $\{e',h',f'\}$ such that 
$h'\in\tilde\g_0$ and $f'\in\tilde\g_1$~\cite{kr71}. Here $\tilde h=e'+f'$ is $SL_2$-conjugate to $h'$ in 
$\langle e',h',f'\rangle\simeq\tri$ and hence in $\g$. Since $e'\in G{\cdot}e$, one also has  
$h'\in G{\cdot}h$. Therefore $\tilde h\in G{\cdot}h$.
Since $[\g^h,\g^h]\simeq (\tri)^{k}$ for some $k$ 
(Theorem~\ref{prop:reg-in-g0}), we have thus detected a semisimple element 
$\tilde h\in\tilde\g_1$ such that
\[
     [\tilde\g^{\tilde h}_0,\tilde\g^{\tilde h}_0]\subset [\g^{\tilde h},\g^{\tilde h}]\simeq (\tri)^{k}.
\]
Since $\tilde h$ is semisimple, the $\tilde G_0$-orbit of $\tilde h$ is closed in $\tilde\g_1$.
For a generic semisimple $x\in\tilde\g_1$, it then follows from Luna's slice theorem~\cite[III.3]{luna} that 
the stabiliser  $\tilde\g^x_0$ 
is $\tilde G_0$-conjugate to a subalgebra of $\tilde\g^{\tilde h}_0$. Therefore, 
$[\g^x,\g^x]=[\tilde\g^x_0,\tilde\g^x_0]\simeq (\tri)^{k'}$ for some $k'\le k$.
\end{proof}

Combining Propositions~\ref{prop:0=2}(1) and \ref{prop:IBN}, we obtain
\begin{cl}    \label{cor:IBN}
Suppose that $e\in\g_{0,\sf reg}\cap\N$ and the $(\sigma,e)$-grading satisfies the condition that
$d_0(0)=d_1(2)$. Then $\mathsf{Sat}(\sigma)$ has only {\sl IBN} and the black nodes of\/ 
$\mathsf{Sat}(\sigma)$ are contained among the zeros of $\eus D(e)$.
\end{cl}

The complete list of $\sigma\in\mathsf{Inv}(\g)$ such that $\g$ is simple and
$d_0(0)=d_1(2)$ for $e\in\g_{0,\sf reg}\cap\N$ is as follows. We point out the pairs $(\g,\g_0)$. 
\\ \indent
1) $\sigma=\vartheta_{\sf max}$ for $\g\ne \mathfrak{sp}_{4n+2}$, i.e.,
$(\sln,\son),(\mathfrak{so}_{2k},\mathfrak{so}_{k}\dotplus\mathfrak{so}_{k}),
(\mathfrak{so}_{2k+1},\mathfrak{so}_{k+1}\dotplus\mathfrak{so}_{k}), 
(\mathfrak{sp}_{4n},\mathfrak{gl}_{2n})$, $(\GR{E}{6},\GR{C}{4}), (\GR{E}{7},\GR{A}{7}), 
(\GR{E}{8},\GR{D}{8}), (\GR{F}{4},\GR{C}{3}\dotplus\GR{A}{1}), (\GR{G}{2},\GR{A}{1}\dotplus\GR{A}{1})$. 
\\ \indent
2) the others: \ $(\mathfrak{so}_{2k},\mathfrak{so}_{k+1}\dotplus\mathfrak{so}_{k-1}),\ 
(\mathfrak{so}_{4k+1}, \mathfrak{so}_{2k+2}\dotplus\mathfrak{so}_{2k-1}), \ 
(\GR{E}{6}, \GR{A}{5}\dotplus\GR{A}{1})$.

\section{A principal inner involution and the $B$-action on $[\ut,\ut]$}
\label{sect:max_int_inv}

\noindent
For a fixed choice of $\h\subset\g$ and $\Delta^+\subset\Delta=\Delta(\g,\h)$, 
the {\it principal\/}  $\BZ$-{\it grading\/} $\g=\bigoplus_{i\in\BZ}\g\lg i\rg$ is defined by the conditions that $\g\lg 0\rg=\h$ and $\g\lg i\rg=\bigoplus_{\gamma: \hot(\gamma)=i} \g_\gamma$ for $i\ne 0$. Then 
$\be=\g\lg {\ge}0\rg$ is a Borel subalgebra, $\ut=[\be,\be]=\g\lg {\ge}1\rg$, and $\ut'=\g\lg {\ge}2\rg$. Accordingly, we set 
$\Delta\lg i\rg:=\{\gamma\in\Delta\mid \hot(\gamma)=i\}$ for $i\ne 0$. In particular,
$\Delta\lg1\rg=\Pi$ is the set of simple roots in $\Delta^+$. (Alternatively, one can say that this 
$\BZ$-grading is determined by a principal $\tri$-triple.) 

For  $\gamma\in\Delta$, let  $e_\gamma\in\g_\gamma$ be a nonzero root vector. 
By a classical result of Kostant~\cite[Theorem\,5.3]{ko59}, a nilpotent element 
$v=\sum_{\gamma\in\Pi}c_\gamma e_\gamma\in\g\lg 1\rg$ is regular in $\g$ if and only if 
$c_\gamma\ne 0$ for all $\gamma\in\Pi$. In this
case,  the $B$-orbit $B{\cdot}v$ is dense in $\ut$.  Define the subspaces
\begin{equation}   \label{eq:max}
  \g_{\lg 0\rg}:=\bigoplus_{i\ \text{even}}\g\lg i\rg, \qquad \g_{\lg 1\rg}:=\bigoplus_{i\ \text{odd}}\g\lg i\rg .
\end{equation}
This provides the $\BZ_2$-grading $\g=\g_{\lg 0\rg}\oplus\g_{\lg 1\rg}$, 
and the corresponding involution of $\g$ is denoted by $\vartheta$.
The $G$-conjugacy class of $\vartheta$ is uniquely determined by two properties (cf.~\cite[Theorem\,2.3]{theta05}: 

1)  $\vartheta$ {\sl is inner} (because $\rk(\g_{\lg 0\rg})=\rk(\g)$), \par  
2) $\g_{\lg 1\rg}=\g_1^{(\vartheta)}$ {\sl contains a regular nilpotent element of\/} $\g$ 
(because $v\in\g\lg 1\rg \subset \g_{\lg 1\rg}$).

\noindent
We say that $\vartheta$ is a {\it principal inner involution} (=\,{\sl PI\/}--{\it involution}). A 
{\sl PI\/}--involution of $\g$ is of maximal rank if and only if $\g\ne\GR{A}{n}, \GR{D}{2n+1},\GR{E}{6}$. 

Here $\be_{\lg 0\rg}=\be\cap\g_{\lg 0\rg}$ is a Borel subalgebra of $\g_{\lg 0\rg}$ and 
$\ut\cap\g_{\lg 0\rg}=\ut'\cap\g_{\lg 0\rg}$ is 
the nilradical of $\be_{\lg 0\rg}$. The roots system of $(\g_{\lg 0\rg},\h)$ is
$\Delta_{ev}:=\{\gamma\in \Delta \mid \hot(\gamma)\ \text{is even}\}$.
Clearly, $\Delta_{ev}^+:=\Delta_{ev}\cap \Delta^+$ is a set of positive roots in $\Delta_{ev}$,
and $\Delta\lg 2\rg$ is a part of the set of simple roots in $\Delta_{ev}^+$. In order to prove that
$\ut'$ contains a dense $B$-orbit, we first attempt to test 
$e=\sum_{\gamma\in\Delta\lg 2\rg }e_\gamma$. However, this does not always work.
If such an $e$ belongs to the dense $B$-orbit in $\ut'$, then 
$[\be,e]=\ut'$ and hence $[\be_{\lg 0\rg},e]=\ut'\cap\g_{\lg 0\rg}$. Therefore, $e$ must be
a regular nilpotent element of $\g_{\lg 0\rg}$ and $\Delta\lg 2\rg $ must be the whole set of simple
roots in $\Delta_{ev}^+$.
If  $\g$ is simple and $\rk (\g)=r$, then $\#\Delta\lg 2\rg=r-1$. 
Hence  $\Delta\lg 2\rg$ is a set of simple roots of $\Delta_{ev}$ if and only if
$\rk([\g_{\lg 0\rg},\g_{\lg 0\rg}])=r-1$, i.e., the centre of
$\g_{\lg 0\rg}$ is one-dimensional.

As is well known, if $(\g,\g_0)$ is a symmetric pair and $\g$ is simple, then either 
\begin{itemize}
\item[$(\mathfrak A)$]  \ $\g_0$ is semisimple (and $\g_1$ is a simple $\g_0$-module); or
\item[$(\mathfrak B)$]  \ $\g_0$ has a one-dimensional centre (and $\g_1$ is a  sum of two simple 
dual $\g_0$-modules).
\end{itemize}

\noindent 
For the {\sl PI\/}--involutions, both possibilities occur. Namely, $\g_{\lg 0\rg}$ is semisimple
{\sl if and only if\/} $\theta$ is fundamental {\sl if and only if\/} $\g$ is not of type $
\GR{A}{r}$ or $\GR{C}{r}$. These two possibilities are considered separately below.

\un{Case $(\mathfrak A)$}. 
Since $\g_{\lg 0\rg}$ is semisimple and $\vartheta$ is inner, we have $\rk(\Delta_{ev})=r$. Hence
there is a unique minimal root $\beta$, with $\hot(\beta)=2k\ge 4$,  
that is not contained in the linear span of $\Delta\lg 2\rg$. Then 
$\tilde\Pi:=\Delta\lg 2\rg\cup\{\beta\}$ is the set of simple roots in $\Delta_{ev}^+$.
(Actually, $\hot(\beta)=4$, see Remark~\ref{rmk:A}, but we do not need this now.)
Accordingly, $\tilde e:=e_\beta+\sum_{\gamma\in\Delta\lg 2\rg }e_\gamma=e_\beta+e$ is a regular nilpotent 
element of $\g_{\lg 0\rg}$.
\begin{thm}   \label{thm:A}
Suppose that $\g_{\lg 0\rg}=\g^{\vartheta}$ is semisimple. As  above,  let 
$\tilde e=e_\beta+e$
be a regular nilpotent element of\/ $\g_{\lg 0\rg}$. Then
\begin{itemize}
\item[\sf (i)] \ $\ut\cap\g^{\tilde e}=\g^{\tilde e}_{\sf nil}$ \ and \ $\dim\g^{\tilde e}_{\sf nil}=2r=2\,\rk(\g)$;
\item[{\sf (ii)}] \  the orbit $B{\cdot}\tilde e$ is dense in $\ut'$, \ $\be\cap\g^{\tilde e}=\ut\cap\g^{\tilde e}$,
and \ $\g^{\tilde e}\subset \g\lg {\ge}-1\rg$.
\end{itemize}
\end{thm}
\begin{proof}  We have $[\be,\tilde e]\subset \ut'=\g\lg{\ge}2\rg$. 
Since the roots in $\tilde\Pi$ are linearly independent, 
\[
[\h,\tilde e]=\bigoplus_{\gamma\in\tilde\Pi} \g_{\gamma}=\g\lg 2\rg\oplus\g_\beta .
\]
By the very definition of $\beta$, the space $[\g\lg 2\rg, e]\subset \g\lg 4\rg$ does not contain $\g_\beta$.
Therefore,
$[\ut,\tilde e]\subset \g\lg{\ge}3\rg\ominus\g_\beta$. 
(The latter is the sum of all root spaces in $\g\lg{\ge}3\rg$ except $\g_\beta$.)
Hence, 
\beq   \label{eq:a}
     \dim[\ut,\tilde e]\le \dim\g\lg{\ge}3\rg-1=\dim\ut-2r .
\eeq 
On the other hand, $\dim\g^{\tilde e}_{\sf nil}\le 2r$ (Theorem~\ref{prop:reg-in-g0}) 
and $\ut\cap\g^{\tilde e} \subset \g^{\tilde e}_{\sf nil}$, since 
$\tilde e$ is almost distinguished in $\g$ by Lemma~\ref{lm:pochti-otm}. Hence
\beq   \label{eq:b} 
  \dim[\ut,\tilde e]= \dim\ut- \dim(\ut\cap\g^{\tilde e})\ge \dim\ut- \dim\g^{\tilde e}_{\sf nil}
\ge \dim\ut-2r .
\eeq 
Consequently, there are equalities everywhere in  \eqref{eq:a} and \eqref{eq:b}, which yields {\sf (i)}. One also has $[\be,\tilde e]=[\h,\tilde e]\oplus [\ut,\tilde e]=\g\lg{\ge}2\rg=\ut'$, i.e., $B{\cdot}\tilde e$ is dense 
in $\ut'$. Hence $\dim(\be\cap\g^{\tilde e})=2r$ and 
$\be\cap\g^{\tilde e}=\ut\cap\g^{\tilde e}$. Finally, $[\g,\tilde e]\supset \g\lg{\ge}2\rg$, hence 
$\g^{\tilde e} \subset \g\lg{\ge}2\rg^\perp=\g\lg {\ge}-1\rg$.
\end{proof}

\begin{cl}
The weighted Dynkin diagram $\eus D(\tilde e)$ contains  $\kappa(\g)$ isolated zeros.
\end{cl}
\begin{proof}
By~Lemma~\ref{lm:pochti-otm}, $\tilde e$ is even in $\g$. Let $\tilde h\in\h$ be a characteristic of 
$\tilde e$ and $\g=\bigoplus_{i\in\BZ} \g(2i)$  the corresponding $\BZ$-grading of $\g$. Using 
Theorem~\ref{prop:reg-in-g0}(\ref{item-1}) with $h=\tilde h$, we obtain
\[
  \rk(\g)+2k=\dim\g^{\tilde h}=\dim\g(0)\ge \dim\g(2)=\dim \g^{\tilde e}_{\sf nil}=2\rk(\g),
\]
where $k$ is the number of zeros in $\eus D(\tilde e)$. Hence $k\ge \left[ \frac{\rk(\g)+1}{2}\right]$,
which proves the assertion for $\g\ne\GR{D}{2n}$ (cf. Remark~\ref{rmk:max-zeros}).
If $\g=\GR{D}{2n}$, then $\g_{\lg 0\rg}=\GR{D}{n}\dotplus\GR{D}{n}$ and 
$\blb(\tilde e)=(2n{-}1,2n{-}1,1,1)$. The description of $\eus D(\tilde e)$ via partitions~\cite[5.3]{CM} 
shows that the number of zeros equals $n+1$, as required. Cf. the diagrams for $\GR{D}{5}$ and 
$\GR{D}{6}$: 
 \quad  
\raisebox{-4ex}
{\begin{tikzpicture}[scale= .6, transform shape]
\tikzstyle{every node}=[circle, draw, fill=white!55]
\node (b) at (1.1,0) {\bf 2};
\node (c) at (2.2,0) {\bf 0};
\node (d) at (3.3,0) {\bf 2};
\node (e) at (4.4,0) {\bf 0};
\node (f) at (3.3,-1.1) {\bf 0};
\foreach \from/\to in {b/c, c/d, d/e, d/f}  \draw[-] (\from) -- (\to);
\end{tikzpicture}
}
\  and \  
\raisebox{-4ex}
{\begin{tikzpicture}[scale= .6, transform shape]
\tikzstyle{every node}=[circle, draw, fill=white!55]
\node (a) at (0,0) {\bf 0};
\node (b) at (1.1,0) {\bf 2};
\node (c) at (2.2,0) {\bf 0};
\node (d) at (3.3,0) {\bf 2};
\node (e) at (4.4,0) {\bf 0};
\node (f) at (3.3,-1.1) {\bf 0};
\foreach \from/\to in {a/b, b/c, c/d, d/e, d/f}  \draw[-] (\from) -- (\to);
\end{tikzpicture}
}.
\end{proof}

\begin{rmk}   \label{rmk:A}
If  $\g_{\lg 0\rg}=\g^\vartheta$ is semisimple,  
then, beside $\Delta\lg 2\rg $, one more root $\beta$ is needed for the basis of $\Delta_{ev}^+$.
For $\GR{D}{r}$ and $\GR{E}{r}$,  this $\beta$ is obtained as follows.
If $\delta\in\Pi$ corresponds to the branching node in the Dynkin diagram
and $\delta_1,\delta_2,\delta_3$ are the adjacent simple roots, then 
\beq    \label{eq:beta1}
\beta=\delta+\delta_1+\delta_2+\delta_3 ,
\eeq  
see the diagram 
\raisebox{-2.5ex}{\begin{tikzpicture}[scale= .6, transform shape]
\draw (0,0.2) node[above] {$\delta_1$} 
        (1.1,0.25) node[above] {$\delta$} 
        (2.2,0.2) node[above] {$\delta_2$} 
        (1.3,-1.1) node[right]  {$\delta_3$} ;
\tikzstyle{every node}=[circle, draw]
\node (a) at (0,0) {};
\node (b) at (1.1,0){};
\node (c) at (2.2,0) {};
\node (d) at (1.1,-1.1) {};
\tikzstyle{every node}=[circle]
\node (g) at (-1.1,0)  {$\dots$};
\node (h) at (3.3,0)  {$\dots$};
\foreach \from/\to in {a/b, b/c,  b/d, g/a, c/h}  \draw[-] (\from) -- (\to);
\end{tikzpicture}}.
Obviously, $\beta$ is not a sum of two roots of height 2. 
\\ In the non-simply-laced cases, it is convenient 
to think that $\delta$ is the unique long simple root that has an adjacent short simple root. Using 
the numbering of $\Pi$ adopted in~\cite{t41}, $\delta$ equals $\ap_{r-1}$ for $\GR{B}{r}$; $\ap_3$ for 
$\GR{F}{4}$, $\ap_2$ for $\GR{G}{2}$. Then one similarly has $\beta=\delta+\sum \frac{(\delta,\delta)}{(\ap,\ap)}\ap$, 
where the sum ranges over all $\ap\in\Pi$ adjacent to $\delta$. That is, 
\beq   \label{eq:beta2} 
\text{$\beta=\begin{cases}  
\ap_{r-2}+\ap_{r-1}+2\ap_r  &  \text{for $\GR{B}{r}$ \ \  ($\ap_r$ is short)}, \\
2\ap_{2}+\ap_{3}+\ap_4   & \text{for $\GR{F}{4}$ \ \ $(\ap_2$ is short)}, \\
3\ap_{1}+\ap_{2}   &  \text{for $\GR{G}{2}$ \ \ $(\ap_1$ is short)}. \end{cases}
$}
\eeq 
Thus, $\beta$ is long and $\hot(\beta)$ equals $4$. Under the usual unfolding procedure 
$\GR{B}{r} \leadsto \GR{D}{r+1}$,
$\GR{F}{4}\leadsto \GR{E}{6}$, and $\GR{G}{2}\leadsto \GR{D}{4}$, this $\delta$ gives rise to the 
simple root associated with the branching node and $\beta$ of~\eqref{eq:beta2} transforms into
$\beta$ of~\eqref{eq:beta1}.

Since $\tilde e$ is regular in $\g_{\lg 0\rg}$, one has $\gamma(\tilde h)=2$ for all 
$\gamma\in \Delta\lg 2\rg \cup\{\beta\}$. This allows us to determine  
$\eus D(\tilde e)$ (for the $G$-orbit of $\tilde e$). Namely, 
the above formulae for $\delta\in\Pi$ and $\beta$ provide the following uniform answer:
$\delta(\tilde h)=2$ and then one put interlacing values $0$ and $2$ on all remaining nodes,
see Tables~\ref{table:odin} and \ref{table:dva}. Clearly, this procedure provides the maximal possible 
number of isolated zeros in the weighted Dynkin diagram.
\end{rmk}

\un{Case $(\mathfrak B)$}.
Here $\rk(\Delta_{ev})=r-1$, $\Delta\lg 2\rg $ is the set of simple roots in $\Delta_{ev}^+$,  
and $e:=\sum_{\gamma\in\Delta\lg 2\rg }e_\gamma\in \g\lg 2\rg$ is already a regular
nilpotent element of $\g_{\lg 0\rg}$.

\begin{thm}   \label{thm:B}
Suppose that\/ $\g_{\lg 0\rg}=\g^\vartheta$ has a one-dimensional centre and $r=\rk(\g)$. Then \\
\centerline{ {\sf (i)} \ $B{\cdot}e$ is dense in $\ut'$, {\sf (ii)} \ $\dim (\g^{e}\cap\ut)=2r-1$, 
{\sf (iii)} \ $2r-2\le \dim \g^{e}_{\sf nil}\le 2r-1$. }
\\ Furthermore, if $e$ is almost distinguished in $\g$,
then  $\dim\g^e_{\sf nil}=2r-1$.
\end{thm}
\begin{proof}
Such a situation occurs only if $\g$ is of type $\GR{A}{r}$ or $\GR{C}{r}$, and it is not hard to check the 
assertion via direct matrix calculations. However, there is a conceptual argument, too.

{\sf (i)} \ We have to prove that $[\be,e]=\ut'$, i.e., $[\g\lg j-2\rg, e]=\g\lg j\rg$ for all $j\ge 2$.
For $j$ even, this already stems from the fact that $e\in\g_{\lg 0\rg}$ is regular nilpotent.
Anyway, our next argument applies to all $j$.

Since $e\in \g\lg 2\rg$, one can take an $\tri$-triple $\{e,h,f\}\subset \g_{\lg 0\rg}$ such that 
$h\in \g\lg 0\rg$ and $f \in \g\lg -2\rg$, see~\cite{vi79}. Set $\g(k)=\{v\in\g\mid [h,v]=kv\}$ so that 
$e\in\g(2)$. However, the grading 
$\{\g(k)\}_{k\in\BZ}$ is not always the principal $\BZ$-grading, i.e., the spaces $\g(i)$ and $\g\lg i\rg$ 
are different. (Actually, $\g\lg k\rg=\g(k)$\  for all $k$ if and only if $\g$ is of type $\GR{A}{2n}$.)
Still, there is a mild relationship, which is sufficient for us. Set 
\[
      \g\lg j,k):=\g\lg j\rg\cap \g(k).  
\]
Since $\eus D(e)$ has only isolated zeros by Theorem~\ref{prop:reg-in-g0}(\ref{item-1}), if 
$\gamma\in\Delta^+$ and $\hot(\gamma)\ge 2$, then $\gamma(h)\ge 1$.  In other words,
if $j\ge 2$ and $\g\lg j,k)\ne 0$, then $k\ge 1$.
Hence $\g\lg j\rg=\bigoplus_{k\ge 1}\g\lg j, k)$.
Since $\g(k)\subset \Ima(\ad e)$ for all $k\ge 1$ (see~{\sf (i)} in page~\pageref{n-i}), we see that 
$\g\lg j\rg $  belongs to $\Ima(\ad e)$ for any $j\ge 2$.
Hence $[\be,e]=\ut'$.

{\sf (ii)} \ By part {\sf (i)}, we have $[\ut,e]= \g\lg {\ge}3\rg$. Hence $\dim (\g^e\cap\ut)=
\dim (\g\lg 1\rg\oplus \g\lg 2\rg)=2r-1$.

{\sf (iii)} \ Combining the $\BZ_2$-grading \eqref{eq:max} and the $\BZ$-grading $\{\g(k)\}_{k\in\BZ}$
determined by $h$, one obtains a mixed grading related to $(\vartheta,e)$. Here 
$\g_{\lg 0\rg}(1)=0$ and $\g_{\lg 0\rg}(2)=\g\lg 2\rg$; hence 
$\dim \g_{\lg 0\rg}(2)=r-1$, Next, 
either $\g_{\lg 1\rg}(2)\ne 0$ or $\g_{\lg 1\rg}(1)\ne 0$ (Lemma~\ref{lm:chet-nechet}); with either dimension at most $r=\dim\g_{\lg 0\rg}(0)$. 
Hence $\dim\g^{e}_{\sf nil}=\dim\g(1)+\dim\g(2)\le 2r-1$.  

On the other hand, $\dim\g^e_{\lg 0\rg}(0)=1$, i.e.,  $\g^e_{\sf red}\cap\g_{\lg 0\rg}$ is one-dimensional. 
This is only possible if  $[\g^e_{\sf red},\g^e_{\sf red}]$ is either trivial or $\tri$.
Consequently, $\dim(\ut\cap \g^e)\le \dim \g^e_{\sf nil} +1$, hence
$\dim \g^e_{\sf nil}\ge 2r-2$. Finally, if $e$ is almost distinguished, then 
$\ut\cap \g^e \subset \g^e_{\sf nil}$ and both spaces actually coincide for dimension reason.
\end{proof}

\begin{rmk}  \label{rmk:B}
{\sf (1)} \ A posteriori, there are more precise assertions related to Theorem~\ref{thm:B}: 

\textbullet\quad If $\g=\GR{A}{2n-1}$ or $\GR{C}{2n-1}$,  then $\g^e_{\sf red}=\tri$ and 
$\dim \g^e_{\sf nil}=2r-2$ (here $r=2n-1$).

\textbullet \quad If $\g=\GR{A}{2n}$ or $\GR{C}{2n}$,  then $\g^e_{\sf red}=\te_1$ (i.e., $e$ is almost distinguished in $\g$) and $\dim \g^e_{\sf nil}=2r-1$ (here $r=2n$).
\\ \indent
{\sf (2)} \ 
The only case in which $e$ is not even in $\g$ is that of $\GR{A}{2n}$ (cf. Example~\ref{ex:not-even} and
Proposition~\ref{thm:cox-even}).
In the remaining cases, $\eus D(e)$ has the maximal possible number of isolated zeros, see Table~\ref{table:dva}. 
\end{rmk}

\noindent
In Tables~\ref{table:odin} and \ref{table:dva}, we gather information on the {\sl PI\/}--involutions 
$\vartheta$ and the nilpotent orbits of $G$ that contain a regular nilpotent element $e_\vartheta$ of 
$\g_{\lg 0\rg}$. For the exceptional (resp. classical) Lie algebras, the orbit $G{\cdot}e_\vartheta$ is 
denoted by the Dynkin-Bala-Carter label~\cite[8.4]{CM} (resp. by the corresponding partition). For the 
exceptional Lie algebras, the structure of $\g^{e_\vartheta}_{\sf red}$ and the numbers 
$\dim\g^{e_\vartheta}_{\sf nil}$ are pointed out in the Tables in~\cite{alela}. For the classical Lie algebras,
this information is being extracted from the partition of $G{\cdot}e_\vartheta$, 
see~\cite[Theorem\,6.1.3]{CM}. 

\begin{table}[ht]
\caption{The {\sl PI\/}--involutions $\vartheta$ and orbits $G{\cdot}e_\vartheta$ for the exceptional Lie algebras}   \label{table:odin}
\begin{center}
\begin{tabular}{>{$}c<{$}| >{$}c<{$} >{$}c<{$} >{$}r<{$} >{$}c<{$} >{$}c<{$} >{$}c<{$}|}
\g & \g_{\lg 0\rg} & G{\cdot}e_\vartheta & \eus D(e_\vartheta)\text{\hspace{6ex}} & \dim\g^{e_\vartheta} & \g_{\sf red}^{e_\vartheta} & \dim\g_{\sf nil}^{e_\vartheta} \\ \hline\hline
\GR{E}{6} & \GR{A}{5}\dotplus\GR{A}{1} & \mathsf E_6(a_3) & \rule{0pt}{4ex} 
\raisebox{-3.2ex}{\begin{tikzpicture}[scale= .55, transform shape]
\tikzstyle{every node}=[circle, draw, fill=white!55]
\node (a) at (0,0) {\bf 2};
\node (b) at (1.2,0) {\bf 0};
\node (c) at (2.4,0) {\bf 2};
\node (d) at (3.6,0) {\bf 0};
\node (e) at (4.8,0) {\bf 2};
\node (f) at (2.4,-1.2) {\bf 0};
\foreach \from/\to in {a/b, b/c, c/d, d/e, c/f}  \draw[-] (\from) -- (\to);
\end{tikzpicture}}  & 12 & \{0\} & 12  
\\
\GR{E}{7} & \GR{A}{7} & \mathsf E_6(a_1) &  \rule{0pt}{4ex}
\raisebox{-3.2ex}{\begin{tikzpicture}[scale= .55, transform shape]
\tikzstyle{every node}=[circle, draw, fill=white!55]
\node (a) at (0,0) {\bf 0};
\node (b) at (1.2,0) {\bf 2};
\node (c) at (2.4,0) {\bf 0};
\node (d) at (3.6,0) {\bf 2};
\node (e) at (4.8,0) {\bf 0};
\node (f) at (6,0) {\bf 2};
\node (g) at (3.6,-1.2) {\bf 0};
\foreach \from/\to in {a/b, b/c, c/d, d/e, e/f, d/g}  \draw[-] (\from) -- (\to);
\end{tikzpicture}} & 15 & \te_1 & 14  
\\
\GR{E}{8} & \GR{D}{8} & \mathsf E_8(a_4) & \rule{0pt}{4ex}
\raisebox{-3.2ex}{\begin{tikzpicture}[scale= .55, transform shape]
\tikzstyle{every node}=[circle, draw, fill=white!55]
\node (h) at (-1.2,0) {\bf 2};
\node (a) at (0,0) {\bf 0};
\node (b) at (1.2,0) {\bf 2};
\node (c) at (2.4,0) {\bf 0};
\node (d) at (3.6,0) {\bf 2};
\node (e) at (4.8,0) {\bf 0};
\node (f) at (6,0) {\bf 2};
\node (g) at (3.6,-1.2) {\bf 0};
\foreach \from/\to in {h/a, a/b, b/c, c/d, d/e, e/f, d/g}  \draw[-] (\from) -- (\to);
\end{tikzpicture}}   & 16 & \{0\} & 16  
\\
\GR{F}{4} & \GR{C}{3}\dotplus\GR{A}{1} & \mathsf F_4(a_2) & \rule{0pt}{4ex} 
\raisebox{-.5ex}{\begin{tikzpicture}[scale= .55, transform shape]
\tikzstyle{every node}=[circle, draw, fill=orange!30]
\node (a) at (0,0) {\bf 2};
\node (b) at (1,0) {\bf 0};
\tikzstyle{every node}=[circle, draw, fill=white!55]
\node (c) at (2.5,0) {\bf 2};
\node (d) at (3.5,0) {\bf 0};
\foreach \from/\to in {a/b,  c/d}  \draw[-] (\from) -- (\to);
\draw (1.4, .07) -- +(.7,0);
\draw (1.4, -.07) -- +(.7,0);
\end{tikzpicture}}  \hspace*{2ex}
& 8 & \{0\} & 8   
\\
\GR{G}{2} & \GR{A}{1}\dotplus\GRt{A}{1} & \mathsf G_2(a_1) & \rule{0pt}{3.5ex} 
\raisebox{-.5ex}{\begin{tikzpicture}[scale= .55, transform shape]
\tikzstyle{every node}=[circle, draw, fill=orange!30]
\node (b) at (1,0) {\bf 0};
\tikzstyle{every node}=[circle, draw, fill=white!55]
\node (c) at (2.5,0) {\bf 2};
\draw (1.4, .1) -- +(.7,0);
\draw (1.43, 0) -- +(.64,0);
\draw (1.4, -.1) -- +(.7,0);
\end{tikzpicture} } \hspace*{5ex}
& 4 & \{0\} & 4\\  \hline
\end{tabular}
\end{center}
\end{table}

\begin{table}[ht]
\caption{The {\sl PI\/}--involutions $\vartheta$ and orbits $G{\cdot}e_\vartheta$ for the classical Lie algebras}   \label{table:dva}
\begin{center}
\begin{tabular}{>{$}l<{$}| >{$}c<{$} >{$}c<{$} >{$}c<{$} >{$}c<{$} >{$}c<{$} >{$}c<{$}|}
\hspace{2ex}\g & \g_{\lg 0\rg} & \blb(G{\cdot}e_\vartheta) & \eus D(e_\vartheta) & \dim\g^{e_\vartheta} & 
\g_{\sf red}^{e_\vartheta} & \dim\g_{\sf nil}^{e_\vartheta} \\ \hline\hline
\GR{B}{2n} & \GR{B}{n}{\dotplus}\GR{D}{n} & (2n{+}1,2n{-}1,1) & \rule{0pt}{3ex} 
\raisebox{-.7ex}{\begin{tikzpicture}[scale= .55, transform shape]
\tikzstyle{every node}=[circle, draw, fill=white!55]
\node (c) at (2.2,0) {\bf 2};
\node (d) at (3.3,0) {\bf 0};
\node (e) at (5.6,0) {\bf 2};
\tikzstyle{every node}=[circle, draw, fill=orange!30]
\node (f) at (7,0) {\bf 0};
\foreach \from/\to in {  c/d}  \draw[-] (\from) -- (\to);
\draw (6, .07) -- +(.6,0);
\draw (6, -.07) -- +(.6,0);
\tikzstyle{every node}=[circle]
\node (g) at (4.5,0)  {$\dots$};
\foreach \from/\to in {d/g, g/e}  \draw[-] (\from) -- (\to);
\end{tikzpicture} } & 4n & \{0\} & 4n
\\
\GR{B}{2n{-}1} & \GR{B}{n{-}1}{\dotplus}\GR{D}{n} & (2n{-}1,2n{-}1,1) & \rule{0pt}{3.5ex} 
\raisebox{-.7ex}{\begin{tikzpicture}[scale= .55, transform shape]
\tikzstyle{every node}=[circle, draw, fill=white!55]
\node (b) at (1.1,0) {\bf 0};
\node (c) at (2.2,0) {\bf 2};
\node (d) at (3.3,0) {\bf 0};
\node (e) at (5.6,0) {\bf 2};
\tikzstyle{every node}=[circle, draw, fill=orange!30]
\node (f) at (7,0) {\bf 0};
\foreach \from/\to in { b/c,  c/d}  \draw[-] (\from) -- (\to);
\draw (6, .07) -- +(.6,0);
\draw (6, -.07) -- +(.6,0);
\tikzstyle{every node}=[circle]
\node (g) at (4.5,0)  {$\dots$};
\foreach \from/\to in {d/g, g/e}  \draw[-] (\from) -- (\to);
\end{tikzpicture} } & 4n{-}1 & \te_1 & 4n{-}2 
\\
\GR{D}{2n} & \GR{D}{n}{\dotplus}\GR{D}{n} & (2n{-}1,2n{-}1,1,1) &  \rule{0pt}{5.5ex} 
\raisebox{-2.7ex}{\begin{tikzpicture}[scale= .5, transform shape]
\tikzstyle{every node}=[circle, draw, fill=white!55]
\node (c) at (2.2,0) {\bf 0};
\node (d) at (3.3,0) {\bf 2};
\node (e) at (5.6,0) {\bf 0};
\node (f) at (6.7,0) {\bf 2};
\node (g) at (7.7,1) {\bf 0};
\node (h) at (7.7,-1) {\bf 0};
\foreach \from/\to in {  c/d, e/f, f/g, f/h}  \draw[-] (\from) -- (\to);
\tikzstyle{every node}=[circle]
\node (g) at (4.5,0)  {$\dots$};
\foreach \from/\to in {d/g, g/e}  \draw[-] (\from) -- (\to);
\end{tikzpicture} } & 4n{+}2 & \te_2 & 4n 
\\
\GR{D}{2n{-}1} & \GR{D}{n{-}1}{\dotplus}\GR{D}{n} & (2n{-}1,2n{-}3,1,1) & \rule{0pt}{5.5ex} 
\raisebox{-2.7ex}{\begin{tikzpicture}[scale= .5, transform shape]
\tikzstyle{every node}=[circle, draw, fill=white!55]
\node (b) at (1.1,0) {\bf 2};
\node (c) at (2.2,0) {\bf 0};
\node (d) at (3.3,0) {\bf 2};
\node (e) at (5.6,0) {\bf 0};
\node (f) at (6.7,0) {\bf 2};
\node (g) at (7.7,1) {\bf 0};
\node (h) at (7.7,-1) {\bf 0};
\foreach \from/\to in { b/c,  c/d, e/f, f/g, f/h}  \draw[-] (\from) -- (\to);
\tikzstyle{every node}=[circle]
\node (g) at (4.5,0)  {$\dots$};
\foreach \from/\to in {d/g, g/e}  \draw[-] (\from) -- (\to);
\end{tikzpicture} } & 4n{-}1 & \te_1 & 4n{-}2 
\\ \hline
\GR{C}{2n-1} & \mathfrak{gl}_{2n-1} & (2n{-}1,2n{-}1) & \rule{0pt}{3.5ex} 
\raisebox{-.7ex}{\begin{tikzpicture}[scale= .55, transform shape]
\tikzstyle{every node}=[circle, draw, fill=orange!30]
\node (b) at (1.1,0) {\bf 0};
\node (c) at (2.2,0) {\bf 2};
\node (d) at (4.5,0) {\bf 0};
\node (e) at (5.6,0) {\bf 2};
\tikzstyle{every node}=[circle, draw, fill=white!55]
\node (f) at (7,0) {\bf 0};
\draw (6, .07) -- +(.6,0);
\draw (6, -.07) -- +(.6,0);
\tikzstyle{every node}=[circle]
\node (g) at (3.4,0)  {$\dots$};
\foreach \from/\to in { b/c, c/g, g/d, d/e}  \draw[-] (\from) -- (\to);
\end{tikzpicture} } & 4n{-}1 & \tri & 4n{-}4 
\\ 
\GR{C}{2n} & \mathfrak{gl}_{2n} & (2n,2n) & \rule{0pt}{3.5ex} 
\raisebox{-.7ex}{\begin{tikzpicture}[scale= .55, transform shape]
\tikzstyle{every node}=[circle, draw, fill=orange!30]
\node (c) at (2.2,0) {\bf 0};
\node (d) at (3.3,0) {\bf 2};
\node (e) at (5.6,0) {\bf 0};
\tikzstyle{every node}=[circle, draw, fill=white!55]
\node (f) at (7,0) {\bf 2};
\draw (6, .07) -- +(.6,0);
\draw (6, -.07) -- +(.6,0);
\tikzstyle{every node}=[circle]
\node (g) at (4.5,0)  {$\dots$};
\foreach \from/\to in {c/d, d/g, g/e}  \draw[-] (\from) -- (\to);
\end{tikzpicture} } & 4n & \te_1 & 4n{-}1 
\\ 
\GR{A}{2n-1} & \mathfrak{gl}_{n}{\dotplus}\mathfrak{sl}_n & (n,n) & \rule{0pt}{3.5ex} 
\raisebox{-.7ex}{\begin{tikzpicture}[scale= .55, transform shape]
\tikzstyle{every node}=[circle, draw, fill=white!55]
\node (b) at (1.1,0) {\bf 0};
\node (c) at (2.2,0) {\bf 2};
\node (d) at (3.3,0) {\bf 0};
\node (e) at (5.6,0) {\bf 2};
\node (f) at (6.7,0) {\bf 0};
\tikzstyle{every node}=[circle]
\node (g) at (4.5,0)  {$\dots$};
\foreach \from/\to in { b/c, c/d, d/g, g/e, e/f}  \draw[-] (\from) -- (\to);
\end{tikzpicture} } & 4n{-}1 & \tri & 4n{-}4 
\\ 
\GR{A}{2n} & \mathfrak{gl}_{n+1}{\dotplus}\mathfrak{sl}_n & (n{+}1,n) & \rule{0pt}{3.5ex} 
\raisebox{-.7ex}{\begin{tikzpicture}[scale= .55, transform shape]
\tikzstyle{every node}=[circle, draw, fill=white!55]
\node (b) at (1.1,0) {\bf 1};
\node (c) at (2.2,0) {\bf 1};
\node (d) at (3.3,0) {\bf 1};
\node (e) at (5.6,0) {\bf 1};
\node (f) at (6.7,0) {\bf 1};
\tikzstyle{every node}=[circle]
\node (g) at (4.5,0)  {$\dots$};
\foreach \from/\to in { b/c, c/d, d/g, g/e, e/f}  \draw[-] (\from) -- (\to);
\end{tikzpicture} } & 4n & \te_1 & 4n{-}1 
\\ 
\hline
\end{tabular}
\end{center}
\end{table}

\begin{rmk}   \label{rmk:collaps}
Since there is a dense $B$-orbit in $\ut'$, the varieties $G\times_B\ut'$ and $G{\cdot}\ut'$ contain dense 
$G$-orbits. But their dimensions can be different, i.e., the natural `collapsing' in the sense of 
Kempf~\cite{ke76} $\tau: G\times_B\ut'\to G{\cdot}\ut'$ is not always generically finite-to-one. Indeed, if 
$B{\cdot}e$ is dense in $\ut'$, then $\dim G{\cdot}\ut'=\dim G{\cdot}e$ is even and 
$\dim(G\times_B\ut')=2\dim\ut-r=\dim\g-2r$. Hence the necessary (but not sufficient) condition is that 
$r$ is even. It is also clear that $d:=\dim(G\times_B\ut')-\dim G{\cdot}\ut'=\dim\g^{e_\vartheta}-2r$.
Hence using Tables~\ref{table:odin} and \ref{table:dva}, we obtain
\\
\centerline{\it $\tau$ is generically finite-to-one if and only if\/ $\rk(\g)$ is even and $\g\ne\GR{D}{2n}$.}
\\
If $\rk(\g)$ is odd, then $d=1$, whereas $d=2$ for $\GR{D}{2n}$.
\end{rmk}

\section{Mixed gradings and divisible orbits}       
\label{sect:divis}

\noindent
For $e\in\N$, let $\frac{1}{2}\eus D(e)$ be the Dynkin diagram equipped with the labels
$\frac{1}{2}\ap(h)$, $\ap\in\Pi$. Following~\cite{divis},
an element $e\in\N$ (orbit $G{\cdot}e\subset\N$) is said to be {\it divisible}, if $\frac{1}{2}\eus D(e)$ is 
again a weighted Dynkin diagram. Equivalently, if $h$ is a characteristic of $e$, then $h/2$ 
is a characteristic of another nilpotent element. Then we write $\edva$ for a  nilpotent element with 
characteristic $h/2$. Hence $\frac{1}{2}\eus D(e)=:\eus D(\edva)$. The notation is suggested by the 
equality 
\[    
        \dim\Ker(\ad e)^2=\dim\Ker(\ad\edva)=\dim\z_\g(\edva) .
\] 
Furthermore, if $\g=\sln$ and $e$ is divisible, then one can take $\edva=e^2$, the usual matrix 
power~\cite[Theorem\,3.1]{divis}.
Clearly, a divisible element must be even.
If $\{e,h,f\}$ is an $\tri$-triple and $\g=\bigoplus_{i\in \BZ}\g(i)$ is the $\BZ$-grading determined by $h$, 
then $e$ is divisible if and only if there is an $x\in \g(4)$ such that $h\in\Ima(\ad x)$. In this case, $e$ is
necessarily even, 
there is an $\tri$-triple of the form $\{x, h/2, \tilde f\}$ with $\tilde f\in\g(-4)$, and one can take
$\edva=x$.
 
Mixed gradings can be used for detecting divisible orbits with good properties. Suppose that $e\in\N$ 
is even and $\sigma(e)=e$ for some $\sigma\in\mathsf{Inv}(\g)$. Consider a mixed grading  
related to $(\sigma, e)$. Our intention is to find a suitable element $\edva\in \g(4)$. 
Of course, this is not always possible, and a sufficient condition is given below. Recall that 
$d_j(i)=\dim\g_j(i)$ and  $d_0(0)\ge d_j(i)$ whenever $i\ne 0$.

\begin{thm}    \label{prop:dostat-del}
As above, let $\g=\bigoplus_{i,j}\g_j(i)$ be a mixed grading related to $(\sigma,e)$, hence 
$e\in\g_0(2)$. Suppose that $d_0(0)=d_1(4)$. Then 
\begin{enumerate}
\item  the orbit $G{\cdot}e$ is divisible and there exists $\edva\in \g_1(4)\subset\g(4)$;
\item $d_0(0)=d_0(2)= d_1(2)=d_1(4)$ and also 
$d_0(4k+2)=d_1(4k+2)$
for all $k\in\BZ$;
\item $\g_1$ does not contain 3-dimensional $\lg e,h,f\rg$-modules,
$\g_0$ is semisimple, and $e$ is distinguished in $\g_0$;
\item both $e$ and $\edva$ are almost distinguished in $\g$.
\end{enumerate}
\end{thm}
\begin{proof} 
(1) \ If $x\in \g_1(4)$, then $[\g_1(-4),x]$ is the orthogonal complement of $\g^x_0(0)$ in $\g_0(0)$. By 
Lemma~\ref{lm:finite-orb}, $G_0(0)$ has an open orbit in $\g_1(4)$, say $\co$. Now, if $x\in\co$, 
then $\g^x_0(0)=\{0\}$ in view of the hypothesis $d_0(0)=d_1(4)$. Hence 
$[\g_1(-4),x]=\g_0(0)$ and there is $y\in \g_1(-4)$ such that $[x,y]=h/2$. Then $\{x,h/2,y\}$ a desired 
$\tri$-triple. 
Thus, every element of the open $G_0(0)$-orbit 
in $\g_1(4)$ can be taken as  $\edva$. This also implies that $e$ must be even in $\g$.

(2) \ Since $h/2$ is a characteristic of $\edva\in\g_1(4)$, the analogue of Eq.~\eqref{eq:mix-inj-surj}
for $\edva$ implies that the mappings $\ad\edva: \g_0(-2)\to \g_1(2)$ and
$\ad\edva: \g_1(-2)\to \g_0(2)$ are bijective. Hence $\dim\g_0(2)=\dim\g_1(2)$.
Combining this with Eq.~\eqref{eq:mix-inj-surj} yields
\[
  d_0(0)\ge d_0(2)= d_1(2)\ge d_1(4)=d_0(0) .
\]
Since the $\BZ$-grading of $\g$ is determined by a characteristic of $e$ and $e\in\g(2)$, 
the $\tri$-theory readily implies that $(\ad e)^m: \g(-m)\to \g(m)$ is a bijection. Likewise, since $h/2$ is a
characteristic of $\edva$ and $\edva\in\g_1(4)$, this implies that 
\[
(\ad\edva)^{2k+1}: \g_0(-4k{-}2)\to \g_1(4k{+}2) \ \text{ and } \ 
(\ad\edva)^{2k+1}: \g_1(-4k{-}2)\to \g_0(4k{+}2)
\] 
are bijections. Therefore, $d_0(4k+2)=d_1(4k+2)$.

(3) \ Now, the equality $d_1(2)= d_1(4)$ means that $\g_1$ contains no $3$-dimensional 
$\lg e,h,f\rg$-modules. While  the equality $d_0(0)= d_0(2)$ implies that $\g_0$ is semisimple 
and $e$ is distinguished in $\g_0$ (cf. the proof of Theorem~\ref{prop:reg-in-g0}(\ref{item-2})). 

(4) \ We know that $\g^e_{\sf red}=\Ker(\ad e)\cap\g(0)$. Since 
$\ad e: \g_0(0) \stackrel{\sim}{\to} \g_0(2)$, we see that $\g^e_{\sf red}\subset \g_1(0)$. Hence
$\g^e_{\sf red}$ is toral. Likewise, $\g^{\edva}_{\sf red}=\Ker(\ad\edva)\cap\g(0)$ and
$\ad\edva: \g_0(0) \stackrel{\sim}{\to} \g_1(4)$. Therefore,
$\g^{\edva}_{\sf red}\subset \g_1(0)$ and $\g^{\edva}_{\sf red}$ is toral.
\end{proof}

\begin{rmk} For a $(\sigma,e)$-grading, the hypothesis $d_0(0)=d_1(4)$ implies that $d_0(0)=d_1(2)$. In
Proposition~\ref{prop:0=2}, we derived from the latter that $G(0){\cdot}e\cap\g_1(2)\ne \varnothing$.
For a divisible orbit $G{\cdot}e$, it follows from~\cite[Theorem\,1]{an} that  
the conditions $G{\cdot}e\cap\g_1\ne \varnothing$ and $G{\cdot}\edva\cap\g_1\ne \varnothing$
are equivalent (cf. also Remark~\ref{rem:Levon}). 
However, a subtle point is that if $G{\cdot}e$ is divisible and
$G{\cdot}e\cap\g_1\ne \varnothing$, then one may not simultaneously have that $e\in\g_0(2)$ and $\edva\in\g_1(4)$. For, starting with $e\in\g_0(2)$, we can certainly find 
$\edva\in\g(4)$, but then we need a stronger condition that $G(0){\cdot}\edva\cap\g_1(4)\ne\varnothing$.
\end{rmk} 
Using Theorem~\ref{prop:dostat-del}, we classify below all the pairs $(\sigma,G{\cdot}e)$ such that 
$d_0(0)=d_1(4)$. 

\begin{thm}   \label{thm:classif-ravenstvo}
For a mixed grading related to  $(\sigma,e)$, the equality $d_0(0)=d_1(4)$ holds 
exactly in the following cases: 
\begin{itemize}
\item[\sf (a)]  $\g=\GR{E}{n}$, $\sigma=\vartheta_{\sf max}$,  
and $e$ is regular in $\g_0$; here 
$\g_0=\GR{C}{4}, \GR{A}{7}, \GR{D}{8}$ and
the Dynkin--Bala--Carter labels of $G{\cdot}e$ are $\mathsf E_6(a_1)$, $\mathsf E_6(a_1)$, 
$\mathsf E_8(a_4)$ for $n=6,7,8$, respectively.
\item[\sf (b)]  $\g=\sln$, $\g_0=\son$ (hence $\sigma$ is also of maximal rank), 
and $\blb(e)=(2m_1+1,\dots,2m_s+1)$  is a partition of $n$ such that $m_{i-1}-m_{i}\ge 2$ for $i\ge 2$. 
In particular, one obtains a regular nilpotent element of\/ $\g_0=\mathfrak{so}_{2m_1+1}$ \emph{(}resp. 
$\mathfrak{so}_{2m_1+2}$\emph{)} if $s=1$ (resp. if $s=2$ and $m_2=0$).
\end{itemize}
\end{thm}
\begin{proof} 
1) The ``only if'' part relies on the explicit description of divisible orbits. For the classical series, 
such a description is given in terms of partitions~\cite[Theorem\,3.1]{divis}, see also  below.
For the exceptional algebras, there is the list of divisible orbits~\cite[Table\,1]{divis}. By 
Theorem~\ref{prop:dostat-del}(4), if $d_0(0)=d_1(4)$, then $e$ and $\edva$ are almost distinguished. 
Let us begin with finding the pairs of orbits $G{\cdot}e$ and $G{\cdot}\edva$ such that
both $e, \edva$ are almost distinguished.

\textbullet\quad For the exceptional algebras, Table~1 in \cite{divis}, together with the known 
information on the reductive part of centralisers~\cite{alela}, shows that there are only three such pairs 
of orbits. This leads to the three $\GR{E}{n}$-cases, see Example~\ref{ex:0=4}. 

\textbullet\quad For $\son$ and $\spn$, an inspection of partitions $\blb(e)$ of the divisible orbits shows that there are no pairs $(e,\edva)$ such that both orbits are almost distinguished. More precisely,

{\bf --}  \ For $\g=\spn$, $e$ is divisible if and only if all parts of $\blb(e)$ are odd and occur pairwise.
This already implies that $e$ is not almost distinguished, cf.~\cite[IV.2.25]{ss} or~\cite[Theorem\,6.1.3]{CM}, where a description of $\g^e_{\sf red}$ is given in terms of $\blb(e)$.

{\bf --}  \ For $\g=\son$, $e$ is divisible if and only if all parts $\lb_i$ are odd and also

\qquad $\left\{\begin{array}{ll}  
\text{if } \lb_{2k-1}=4m+3, & \text{then } \lb_{2k}=4m+3 ; \\
\text{if } \lb_{2k-1}=4m+1>1, & \text{then } \lb_{2k}\in \{4m+1,4m-1\} ; \\
\text{if } \lb_{2k-1}=1, & \text{then there is no further conditions} .\\
\end{array}\right.$
\\[.6ex]
The partition $\blb(\edva)$ is obtained by the following rule applied to each pair 
$(\lb_{2k-1},\lb_{2k})$ of consecutive parts of $\blb(e)$:

\qquad $\left\{\begin{array}{lcl}  
(\dots,4m+3,4m+3,\dots) &  \mapsto & (\dots,2m+2,2m+2,2m+1,2m+1,\dots); \\
(\dots,4m+1,4m+1,\dots) &  \mapsto & (\dots,2m+1,2m+1,2m,2m,\dots); \\
(\dots,4m+1,4m-1,\dots) &  \mapsto & (\dots,2m+1,2m,2m,2m-1,\dots).\\
\end{array}\right.$
\\[.6ex]
In all cases, $\blb(\edva)$ has at least two equal {\bf even} parts, which yields a subalgebra 
$\tri=\mathfrak{sp}_2$ in $\z_\g(\edva)$.

\textbullet\quad Let $\g=\sln$ and $\blb(e)= (\lb_1,\dots,\lb_s)$, where $\lb_1\ge \dots \ge\lb_s$. Then $e$ is 
divisible if and only if all $\lb_i$'s are odd, say $\lb_i=2m_i+1$; and  $e$ is almost distinguished if and 
only if all $\lb_i$'s are different, i.e., $m_i> m_{i+1}$. Next, the set of parts of $\blb(\edva)$ is
$\{m_i+1,m_i\mid i=1,\dots,s\}$~\cite[Theorem\,3.1]{divis}. Since $\edva$ is almost distinguished as well, 
all these parts must be different, too. Hence 
$m_i>m_{i+1}+1$.
\\ \indent Since $\g_0$ has to be semisimple by Theorem~~\ref{prop:dostat-del}(3), $\sigma$ is outer, 
i.e., $\g_0=\mathfrak{sp}_n$ (if $n$ is even) or $\son$. But a partition with different odd parts does not 
correspond to a nilpotent orbit in $\mathfrak{sp}_n$~\cite[Theorem\,5.1.3]{CM}. Hence $\g_0=\son$. 

2) The "if" part follows by direct calculations, cf. Example~\ref{ex:0=4} below. Let us give some details for 
the case of $\sln=\slv$. Let $\sfr_i$ denote the simple $\tri$-module of dimension $i+1$.
If $\blb(e)=(2m_1+1,\dots,2m_s+1)$, $\g_0=\son$, and $\tri\simeq\lg e,h,f\rg\subset\g_0$, then
$\BV=\sfr_{2m_1}+\dots +\sfr_{2m_s}$ as $\tri$-module. Therefore
\begin{gather*}
\g_0=\wedge^2(\sfr_{2m_1}+\dots +\sfr_{2m_s})=
\wedge^2\sfr_{2m_1}+\dots +\wedge^2\sfr_{2m_s}+\sum_{i<j}\sfr_{2m_i}\otimes\sfr_{2m_j}, \\
\g_1=\eus S^2(\sfr_{2m_1}+\dots +\sfr_{2m_s})-\sfr_0=
\eus S^2\sfr_{2m_1}+\dots +\eus S^2\sfr_{2m_s}+\sum_{i<j}\sfr_{2m_i}\otimes\sfr_{2m_j}-\sfr_0 .
\end{gather*}
Using the Clebsch--Gordan formula (see e.g.~\cite[3.2.4]{spr77}), we see that if $m_i-m_{i+1}\ge 2$, then
{\it\bfseries (1)} \ the total number of $\tri$-modules in $\g_0$ equals the number of nontrivial $\tri$-modules in $\g_1$, i.e., $d_0(0)=d_1(2)$, and {\it\bfseries (2)} \ $\g_1$ contains no $\sfr_2$, i.e., $d_1(2)=d_1(4)$. \\ \indent
We also notice that here $d_0(0)=\sum_{j=1}^s (2j-1)m_j+\binom{s}2$ and $d_1(0)-d_0(0)=s-1$.  
\end{proof}

\begin{rmk}
The constraints on $\blb$ for $\sln$ exclude exactly the cases with $n=2,4$. This means that only those $n$ are allowed for which $(\sln)^\sigma=\son$ is again simple.
\end{rmk}
\begin{ex}   
\label{ex:0=4}
{\sf (1)} \ Let us provide the numbers $d_j(i)=\dim\g_j(i)$ with $i\ge 0$ for the $\GR{E}{n}$-cases in
Theorem~\ref{thm:classif-ravenstvo}.

\begin{center}
 $\begin{array}{c}\g=\GR{E}{6}, \g_0=\GR{C}{4}\\ 
   G{\cdot}e=\mathsf E_6(a_1)\end{array}$: \qquad \begin{tabular}{c|ccccccccc}
   $i$              & 0 & 2 & 4 & 6 & 8 & 10 & 12 & 14 &16 \\ \hline
$d_0(i)$ & \rule{0pt}{2.5ex}\fbox{\color{violet}4} & 4 & 3 & 3 & 2 &  2 & 1 & 1 & -- \\  
$d_1(i)$ & 4 & 4 & \fbox{\color{violet}4} & 3 & 3 &  2 & 1 & 1 & 1 \\  
\end{tabular}
\end{center}

\noindent
Here $\g_0\simeq\sfr_2+\sfr_6+\sfr_{10}+\sfr_{14}$ and
$\g_1\simeq\sfr_4+\sfr_8+\sfr_{10}+\sfr_{16}$ as $\lg e,h,f\rg$-modules.
\\ \rule{4cm}{.5pt} 
\begin{center}
$\begin{array}{c}\g=\GR{E}{7}, \g_0=\GR{A}{7}\\ 
   G{\cdot}e=\mathsf E_6(a_1)\end{array}$
 : \qquad \begin{tabular}{c|ccccccccc}
   $i$              & 0 & 2 & 4 & 6 & 8 & 10 & 12 & 14 &16 \\ \hline
$d_0(i)$ & \rule{0pt}{2.5ex}\fbox{{\color{violet}7}} & 7 & 6 & 5 & 4 &  3  &  2 &  1 & -- \\  
$d_1(i)$ & 8 & 7 & \fbox{\color{violet}7} & 5 & 5 &  3 &  2 & 1 & 1 \\  
\end{tabular}
\end{center}

\noindent
Here $\g_0\simeq \sum_{i=1}^7 \sfr_{2i}$ and
$\g_1\simeq\sfr_0+2\sfr_4+2\sfr_8+\sfr_{10}+\sfr_{12} +\sfr_{16}$ as $\lg e,h,f\rg$-modules.
\\ \rule{4cm}{.5pt} 
\begin{center}
 $\begin{array}{c}\g=\GR{E}{8}\\ \g_0=\GR{D}{8} \\ G{\cdot}e=\mathsf E_8(a_4)
 \end{array}$: \quad \begin{tabular}{c|ccccccccccccccc}
   $i$              & 0 & 2 & 4 & 6 & 8 & 10 & 12 & 14 &16 &18 & 20 & 22 & 24 & 26 & 28 \\ \hline
$d_0(i)$ & \rule{0pt}{2.5ex}
\fbox{{\color{violet}8}} & 8 & 7 & 7 & 6 & 6 & 5 & 5 & 3 & 3 & 2 & 2 & 1 & 1 & -- \\  
$d_1(i)$ & 8 & 8 & \fbox{\color{violet}8} & 7 & 7 & 6 & 5 & 5 & 4 & 3 & 2 & 2 & 1 & 1 & 1\\  
\end{tabular}
\end{center}

\noindent
Here $\g_0\simeq\sfr_2+\sfr_6+\sfr_{10}+2\sfr_{14}+\sfr_{18}+\sfr_{22}+\sfr_{26}$ \\ and 
$\g_1\simeq\sfr_4+\sfr_8+\sfr_{10}+\sfr_{14}+\sfr_{16}+\sfr_{18}+\sfr_{22}+\sfr_{28}$ as $\lg e,h,f\rg$-modules.

{\sf (2)} \ The following is a sample of calculations for the $\sln$-case.
\begin{center}
 \begin{tabular}{c} $\g=\mathfrak{sl}_{2n}$,\\ $e\sim (2m+1,2k+1)$ \\ $m{+}k=n{-}1,\ m{-}k{>}1$
 \end{tabular}: 
\quad 
 \begin{tabular}{c|cccccc}
   $i$              & 0 & 2 & 4 & $\cdots$ & $4m{-}2$& $4m$  \\ \hline
$d_0(i)$ & \rule{0pt}{2.5ex}\fbox{\color{violet}m+3k+1} & m+3k+1 & m+3k-1 & $\cdots$ & $1$ & -- \\  
$d_1(i)$ & m+3k+2 & m+3k+1 & \fbox{\color{violet}m+3k+1} & $\cdots$ & $1$ & 1  \\  
\end{tabular}
\end{center}
\end{ex}

\begin{rmk}  \label{rem:raznitsa=1}
There are many examples of mixed gradings related to some $(\sigma, e)$ such that 
$d_0(0)> d_1(4)$ and $e\in\g_0(2)$ is divisible in $\g$. Hence there is $\edva\in\g(4)$.
However, it is then not always the case that  
$G{\cdot}\edva\cap\g_1(4)\ne \varnothing$. 
\end{rmk}

\section{New involutions from old ones}      
 \label{sect:new-invol}

\noindent
Let $\sigma$ be an involution of a simple Lie algebra $\g$ and $e$ a regular nilpotent element of 
$\g^\sigma=\g_0$. If $e$ remains even in $\g$, then one defines
another involution having nice properties, which is denoted by $\Upsilon(\sigma)$ or $\check\sigma$. 
It is always assumed below that $e\in\g_{0,{\sf reg}}$ is even in $\g$. By Proposition~\ref{thm:cox-even}, 
this only excludes the inner involutions of $\mathfrak{sl}_{2n+1}$.

{\bf Definition/construction of }$\Upsilon(\sigma)=\check\sigma$. \\
Let $\g=\g_0\oplus\g_1$ be the $\BZ_2$-grading corresponding to $\sigma$ and $\{e,h,f\}\subset\g_0$ 
a principal $\tri$-triple in $\g_0$. Consider the $\BZ$-grading of $\g$ determined by 
$h$ and, assuming that $e$ is even in $\g$,  set 
\beq    \label{eq:sigma2}
   \g_{\{0\}}=\bigoplus_{i \text{ even}}\g(2i), \ \g_{\{1\}}=\bigoplus_{i \text{ odd}}\g(2i) .
\eeq 
Then $\g=\g_{\{0\}}\oplus \g_{\{1\}}$ and 
$\check{\sigma}$ is the involution associated with this $\BZ_2$-grading, i.e., $\g_{\{i\}}=\g_i^{(\check\sigma)}$.

{\bf First properties of $\Upsilon$ (the passage $\sigma\mapsto \check\sigma$)}: \\
{\sf (1)} \quad $\check{\sigma}$ is inner (for, $\g^{\check\sigma}=\g_{\{0\}}\supset \z_\g(h)$, and the latter contains a Cartan subalgebra of $\g$).
\\
{\sf (2)}  \quad  $e\in\g_0(2)\subset \g_0$ and therefore $e\in \g_{\{1\}}$. 
\\
{\sf (3)}  \quad The involutions $\sigma$ and $\check{\sigma}$ commute; hence $\sigma\check\sigma$
is also an involution.
\\
{\sf (4)}  \quad  $\sigma$ and $\sigma\check\sigma$ belong to the same connected component of the 
group ${\sf Aut}(\g)$; therefore, $\sigma$ is inner if and only if $\sigma\check\sigma$ is inner.

For $\g\ne\mathfrak{sl}_{2n+1}$, we think of $\Upsilon$ as a map from $\mathsf{Inv}(\g)$ to 
the set of inner involutions of $\g$. \\
Let $\g=\g_{[0]}\oplus\g_{[1]}$ be the $\BZ_2$-grading corresponding to 
$\sigma\check\sigma$, i.e., $\g_{[i]}=\g_i^{(\sigma\check\sigma)}$. Then
\beq   \label{eq:sigma3}
   \g^{\sigma\check\sigma}=\g_{[0]}=\bigoplus_{i\in \BZ}\g_{\bar i}(2i)\quad  \text{and } \quad 
   \g_{[1]}=\bigoplus_{i\in \BZ}\g_{\ov{i+1}}(2i),
\eeq
where $\bar i$ is the image of $i$ in $\BZ/2\BZ$. In particular, $e\in \g_{[1]}$.

\begin{prop}    \label{prop:IBN2}
For any $\sigma\in {\sf Inv}(\g)$, the Satake diagrams $\mathsf{Sat}(\check\sigma)$ and 
$\mathsf{Sat}(\sigma\check\sigma)$ have {\sl IBN}. Moreover, the set of black nodes of either of them
is contained in the set of zeros of\/ $\eus D(e)$.
\end{prop}
\begin{proof}
This readily follows from Proposition~\ref{prop:IBN}, since
$e\in\g_1^{(\check\sigma)}$ and $e\in\g_1^{(\sigma\check\sigma)}$.
\end{proof}

The Satake diagrams with {\sl IBN}\/ can be extracted from~\cite[Table~4]{t41} or tables in~\cite{spr87}. 
One can also use information on generic stabilisers for $\g_0$-modules $\g_1$ (cf. the proof of 
Prop.~\ref{prop:IBN}). For instance, let $\sigma_{n,m}$ be an involution of 
$\g=\mathfrak{so}_{n+m}$ such that $\g^{\sigma_{n,m}}=\mathfrak{so}_n\dotplus\mathfrak{so}_m$.
The $\g_0$-module $\g_1$ is isomorphic to $\BC^n\otimes\BC^m$ and  a generic 
stabiliser is $\mathfrak{so}_{|n-m|}$. Therefore, $\mathsf{Sat}(\sigma_{n,m})$ has only {\sl IBN\/} if and only if 
$|n-m|\le 4$. Another example is that the Satake diagram of the {\sl PI\/}--involution $\vartheta$ has no 
black nodes at all (but has some arrows if $\g=\GR{A}{n},\,\GR{D}{2n+1},$ or $\GR{E}{6}$). 

Recall from~\cite{Dy52} that a subalgebra $\ka$ of $\g$ is said to be {\it regular}, if it is normalised by a
Cartan subalgebra. For $\sigma$ inner,  $\g^\sigma=\g_0$ is a regular reductive subalgebra. Therefore,
if $e\in\g_{0,\sf reg}\cap\N$, then $[\g_0,\g_0]$ is a ``minimal including regular subalgebra'' (=\,\textsf{MIRS}) for 
$G{\cdot}e$ (in the terminology of \cite{Dy52}). For $\g$ exceptional and $e\in\N$, Dynkin explicitly
determined all \textsf{MIRS} for $G{\cdot}e$, see Tables~16-20 in~\cite{Dy52}. (A few 
inaccuracies has been corrected in~\cite{alela}). Therefore, it is rather easy to determine $G{\cdot}e$ and
$\eus D(e)$ for the inner involutions.

\begin{ex}   \label{ex:E7|D6A1}
For $(\g,\g_0)=(\GR{E}{7}, \GR{D}{6}\dotplus\GR{A}{1})$ and $e\in \g_{0,{\sf reg}}$, the subalgebra
$\GR{D}{6}\dotplus\GR{A}{1}$ is a \textsf{MIRS} of $\g$ for $G{\cdot}e$. 
Then browsing Table~19 in \cite{Dy52} or Table~5 in \cite{alela} one
finds $\eus D(e)$. This orbit is denoted nowadays as
$\mathsf E_7(a_3)$ and 
$\eus D(e)=$ \raisebox{-3.6ex}{\begin{tikzpicture}[scale= .53, transform shape]
\tikzstyle{every node}=[circle, draw, fill=white!55]
\node (a) at (0,0) {\bf 2};
\node (b) at (1.2,0) {\bf 2};
\node (c) at (2.4,0) {\bf 0};
\node (d) at (3.6,0) {\bf 2};
\node (e) at (4.8,0) {\bf 0};
\node (f) at (6,0) {\bf 2};
\node (g) at (3.6,-1.2) {\bf 0};
\foreach \from/\to in {a/b, b/c, c/d, d/e, e/f, d/g}  \draw[-] (\from) -- (\to);
\end{tikzpicture}}\ .
Since the Satake diagram for $(\g,\g_0)$ is 
$\text{\begin{picture}(87,22)(14,0)
\setlength{\unitlength}{0.016in} 
\multiput(23,15)(15,0){5}{\line(1,0){9}} \put(65,0){\circle*{5}}
\multiput(65,15)(15,0){3}{\circle{5}}
\put(65,3){\line(0,1){9}}
\multiput(20,15)(30,0){2}{\circle*{5}} \put(35,15){\circle{5}}
\end{picture}\quad }\ $, 
we see that here $G{\cdot}e\cap\g_1=\varnothing$ (cf. Proposition~\ref{prop:IBN}).
\end{ex}

By \cite[Theorem\,6]{kr71}, there is a unique maximal nilpotent orbit $G{\cdot}e'$ such that 
$G{\cdot}e'\cap\g_1\ne\varnothing$. Here $\eus D(e')$ is obtained from $\mathsf{Sat}(\sigma)$ as follows.
Set $\ap(h')=2$ (resp. $\ap(h')=0$)
if $\ap\in\Pi$ represents a white (resp. black) node of $\mathsf{Sat}(\sigma)$.

It is not hard to determine the conjugacy class of both $\Upsilon(\sigma)=\check\sigma$ and
$\sigma\check\sigma$ for {\bf all} $\sigma\in \mathsf{Inv}(\g)$, excluding the inner involutions of 
$\sltno$. First, one describes the structure of $\g_0$ and $\g_1$ as $\lg e,h,f\rg$-module. For the 
classical $\g$, one can use the partition $\blb(e)$; while if $\g$ is exceptional, then tables of \cite{law95} 
are helpful. This allows us easily to 
compute $\dim\g_{\{0\}}-\dim\g_{\{1\}}$ and $\dim\g_{[0]}-\dim\g_{[1]}$. And the conjugacy class of 
$\tilde\sigma$ such that $\mathsf{Sat}(\tilde\sigma)$ 
has only {\sl IBN}\/ is uniquely determined by $\dim\g^{\tilde\sigma}$, see below.

\begin{lm}   \label{lm:new-differ}
If\/ $\g_0=\sum_{k\ge 0} m_{0,k}\sfr_{2k}$ and  $\g_1=\sum_{k\ge 0} m_{1,k}\sfr_{2k}$ as 
$\lg e,h,f\rg$-modules, then
\begin{gather}   \label{eq:differ-2}
 \dim\g_{\{0\}}-\dim\g_{\{1\}}=\sum_{k\ge 0}(-1)^k m_{0,k}+ \sum_{k\ge 0}(-1)^k m_{1,k} \quad 
 (\text{for }\check\sigma);
\\    \label{eq:differ-3}
\dim\g_{[0]}-\dim\g_{[1]}=\sum_{k\ge 0}(-1)^k m_{0,k}+ \sum_{k\ge 0}(-1)^{k+1} m_{1,k} \quad 
(\text{for }\sigma\check\sigma) . 
\end{gather}
\end{lm}
\begin{proof}
It follows from \eqref{eq:sigma2} and \eqref{eq:sigma3} that
$\g_0\cap \g_{\{0\}}=\g_0\cap \g_{[0]}$ and the $h$-eigenvalues here are multiples of $4$. Therefore,
if $\sfr_{2k}\subset\g_0$, then 
\[
    \dim (\sfr_{2k}\cap \g_{\{0\}})=\dim (\sfr_{2k}\cap \g_{[0]})=\begin{cases}
    k, & \text{ if $k$ is odd} \\ k+1, & \text{ if $k$ is even} \end{cases}.
\]
On the other hand,  $\g_1\cap \g_{\{0\}}=\g_1\cap \g_{[1]}$. Therefore
if $\sfr_{2k}\subset\g_1$, then
\[
    \dim (\sfr_{2k}\cap \g_{\{0\}})=\dim (\sfr_{2k}\cap \g_{[1]})=\begin{cases}
    k, & \text{ if $k$ is odd} \\ k+1, & \text{ if $k$ is even} \end{cases}.
\]
This means that  $\sfr_{2k}\subset \g_0$ contributes `$(-1)^k$' to both 
differences in the lemma, while $\sfr_{2k}\subset \g_1$ contributes `$(-1)^k$' to~\eqref{eq:differ-2} and
`$(-1)^{k+1}$' to~\eqref{eq:differ-3}.
\end{proof}
 
\begin{rmk}    \label{rmk:differ-classic}
For classical Lie algebras, the structure of $\g_0$ and $\g_1$ as $\lg e,h,f\rg$-modules usually involves 
terms of the form $\sfr_m\otimes\sfr_{m+2k}$ or $\eus S^2 \sfr_m$ or $\wedge^2 \sfr_m$. Then it is 
easy to compute the contribution of the whole such aggregates to the differences in 
Lemma~\ref{lm:new-differ}, see Example~\ref{ex:Dn-hermit} below. 
\end{rmk}

\begin{lm}    \label{lm:raznost-IBN}
If the Satake diagram of $\sigma\in \mathsf{Inv}(\g)$ has only {\sl IBN}, then 
\[
 \dim\g_1-\dim\g_0=\rk(\g)-2{\cdot}\#\{\text{arrows}\}- 4{\cdot}\#\{\text{black nodes}\} .
\]
Furthermore, for given $\g$, this quantity distinguishes different Satake diagrams with {\sl IBN}.
\end{lm}
\begin{proof}
For $x\in \g_1$, the value $\dim\g^x_1-\dim\g^x_0$ does not depend on 
$x$~\cite[Proposition\,5]{kr71}. Hence the value for $x=0$, which we need, can be computed via a 
generic semisimple $x\in\g_1$. Then $\g^x$ is a Levi subalgebra of $\g$ whose semisimple part 
corresponds to the subdiagram of black nodes of $\mathsf{Sat}(\sigma)$ and
$\g^x_1$ is Cartan subspace of $\g_1$. Therefore
\[
\dim\g^x_1=\#\{\text{white nodes}\}-\#\{\text{arrows}\}=
\rk(\g)-\#\{\text{black nodes}\}-\#\{\text{arrows}\}
\] 
and $\dim\g^x_0=\dim [\g^x,\g^x]+\#\{\text{arrows}\}$. Finally,
if $\mathsf{Sat}(\sigma)$ has only {\sl IBN}, then $[\g^x,\g^x]\simeq (\tri)^k$, where 
$k=\#\{\text{black nodes}\}$.

Now, a straightforward verification shows that, for a fixed $\g$ and the Satake diagrams with only {\sl IBN}, the quantities \ \ 
$\#\{\text{arrows}\}+2{\cdot}\#\{\text{black nodes}\}$ are all different.
\end{proof}
\begin{ex}  \label{ex:Dn-hermit}
Let $\g=\sone=\mathfrak{so}(\BV)$ and let $\sigma\in \mathsf{Inv}(\sone)$ be such that $\g^\sigma=\gln$. Then $\sigma$ is inner, $\dim\g_0-\dim\g_1=n$, and $\mathsf{Sat}(\sigma)$ is
\begin{gather*}
\text{\begin{picture}(150,24)(0,5)
\put(10,8){\circle*{6}}    
\put(30,8){\circle{6}}     
\put(13,8){\line(1,0){14}}
\multiput(70,8)(40,0){2}{\circle*{6}} 
\put(90,8){\circle{6}}
\multiput(73,8)(20,0){2}{\line(1,0){14}}
\multiput(130,-2)(0,20){2}{\circle{6}}
\put(113,10){\line(2,1){13}}
\put(113,6){\line(2,-1){13}}
\multiput(33,8)(29,0){2}{\line(1,0){5}}   
\put(42,5){$\cdots$}
{\color{blue}\put(137,8){\oval(20,20)[r]}}
{\color{blue}\put(133,18){\vector(-1,0){2}}}
{\color{blue}\put(130,-2){\vector(-1,0){2}}}
\end{picture}, \ if $n$ is odd}; \\
\text{ \begin{picture}(150,30)(0,3)
\put(10,8){\circle*{6}}    
\put(30,8){\circle{6}}     
\put(13,8){\line(1,0){14}}
\put(70,8){\circle*{6}} 
\put(90,8){\circle{6}}
\put(73,8){\line(1,0){14}}
\put(110,-2){\circle{6}}
\put(110,18){\circle*{6}}
\put(93,10){\line(2,1){13}}
\put(93,6){\line(2,-1){13}}
\multiput(33,8)(29,0){2}{\line(1,0){5}}   
\put(42,5){$\cdots$}
\end{picture}, \ if $n$ is even}.
\end{gather*}
These different Satake diagrams suggest that $\check\sigma$ and $\sigma\check\sigma$ might also
depend on the parity of $n$. Indeed, if $e\in\g_{0,\sf reg}\cap\N$, then $\blb(e)=(n,n)$ and 
$\BV=\sfr_{n-1}\oplus\sfr_{n-1}$. Therefore
$\g_0=\sfr_{n-1}\otimes\sfr_{n-1}$ and $\g_1=2{\cdot}\wedge^2\sfr_{n-1}$. The number of simple
$\tri$-modules in $\g_0$ (resp. $\g_1$) equals $n$ (resp. $2{\cdot}[n/2]$).

{\sf (i)} \  Suppose that $n=2m+1$. Using Lemma~\ref{lm:new-differ} and the Clebsch--Gordan 
formulae, we see that the contribution of all $\tri$-modules in $\g_0$ (resp. $\g_1$) to the difference
$\dim\g_{\{0\}}-\dim\g_{\{1\}}$ equals $1$ (resp. $-2m=-2{\cdot}[n/2]$). Hence
$\dim\g_{\{0\}}-\dim\g_{\{1\}}=1-2m=-(\rk\g-2)$. By Lemma~\ref{lm:raznost-IBN},
this can only happen if $\mathsf{Sat}(\check\sigma)$ has one arrow and no black nodes.
Hence $\check\sigma=\vartheta$ is a {\sl PI}-involution with $\g^\vartheta\simeq \mathfrak{so}_{2m+2}\dotplus
\mathfrak{so}_{2m}=\mathfrak{so}_{n+1}\dotplus \mathfrak{so}_{n-1}$.  
\\ \indent
For $\sigma\check\sigma$, the contribution from the $\tri$-modules in $\g_1$ is counted with the opposite sign, i.e., $\dim\g_{[0]}-\dim\g_{[1]}=1+2m=\rk(\g)$. This means that $\sigma$ and $\sigma\check\sigma$
are $G$-conjugate.

{\sf (ii)} \  Suppose that $n=2m$. Then the $\tri$-modules in $\g_0$ make zero contribution
to either of the differences, while the contribution from $\tri$-modules in $\g_1$ to~\eqref{eq:differ-2}
and~\eqref{eq:differ-3} equals $n$ and $-n$, respectively. Hence
$\dim\g_{\{0\}}-\dim\g_{\{1\}}=n=\rk\g$ and $\dim\g_{[0]}-\dim\g_{[1]}=-n=-\rk(\g)$.
This means that $\check\sigma$ and $\sigma$ are conjugate, while
$\sigma\check\sigma$ is of maximal rank, i.e.,
$\g^{\sigma\check\sigma}\simeq \son\dotplus\son$.
\end{ex}

\begin{ex}  \label{ex:vartheta}
For the {\sl PI}--involution $\vartheta$ and $\g\ne \GR{A}{2n}$, we have

\textbullet \  \ $\Upsilon(\vartheta)=\check\vartheta\sim\vartheta$ unless
$\g\in\{\GR{A}{4n+1}, \GR{B}{4n+1}, \GR{C}{2n+1}, \GR{D}{4n+2}, \GR{E}{7} \}$;

\textbullet \ \ If $\g=\GR{A}{4n+1}$, then $\g^\vartheta=\GR{A}{2n}\dotplus\GR{A}{2n}\dotplus\te_1$ and
$\g^{\Upsilon(\vartheta)}=\GR{A}{2n+1}\dotplus\GR{A}{2n-1}\dotplus\te_1$;

\textbullet \ \ If $\g=\GR{B}{4n+1}$, then $\g^\vartheta=\GR{B}{2n}\dotplus\GR{D}{2n+1}$ and
$\g^{\Upsilon(\vartheta)}=\GR{B}{2n+1}\dotplus\GR{D}{2n}$;

\textbullet \ \ If $\g=\GR{C}{2n+1}$, then $\g^\vartheta=\mathfrak{gl}_{2n+1}$ and
$\g^{\Upsilon(\vartheta)}=\GR{C}{n+1}\dotplus\GR{C}{n}$;

\textbullet \ \ If $\g=\GR{D}{4n+2}$, then $\g^\vartheta=\GR{D}{2n+1}\dotplus\GR{D}{2n+1}$ and
$\g^{\Upsilon(\vartheta)}=\GR{D}{2n+2}\dotplus\GR{D}{2n}$;

\textbullet \ \ If $\g=\GR{E}{7}$, then $\g^\vartheta=\GR{A}{7}$ and
$\g^{\Upsilon(\vartheta)}=\GR{D}{6}\dotplus\GR{A}{1}$.
\\[.6ex]
For computations in $\GR{E}{7}$ and $\GR{E}{8}$, one can use data from Example~\ref{ex:0=4}(1).
Further calculations show that one always has $\Upsilon^2(\vartheta)\sim\vartheta$. Note however that,
in general,  
$\Upsilon^2(\sigma)$ has no relation to $\sigma\check\sigma$.
\end{ex}
Since $\sigma$ and $\check{\sigma}$ commute, they determine a $\BZ_2\times\BZ_2$-grading or 
{\it quaternionic decomposition} $\g=\bigoplus_{i,j\in\BZ_2\times\BZ_2}\g_{ij}$. We refer 
to~\cite{hg01,TG13,ANT13} for various invariant-theoretic results related to this structure. 
However, it is worth stressing that our construction gives an ordered triple of involutions. For, starting 
from $\check\sigma$ or $\sigma\check{\sigma}$, one usually obtains another triple and a different 
$\BZ_2\times\BZ_2$-grading. 
\\ \indent
As a complement to Theorem~\ref{prop:dostat-del}, we have

\begin{thm}   \label{thm:conjugate}
Suppose that $\sigma\in {\sf Inv}(\g)$, $e\in\g_0\cap\N$ is regular in $\g_0$, and 
$\g=\bigoplus_{i,j\in\BZ\times\BZ_2}\g_j(2i)$ is the related $(\sigma,e)$-grading. If 
$d_0(4k+2)=d_1(4k+2)$ for all $k\in\BZ$, then $\dim\g^\sigma=\dim\g^{\sigma\check\sigma}$.
Moreover, if\/ $\g_0$ is semisimple, then $\sigma$ and $\sigma\check{\sigma}$ are conjugate involutions.
\end{thm}
\begin{proof}
By Proposition~\ref{thm:cox-even}, $e$ remains even in $\g$, which justifies our notation for the mixed grading. Recall that
\\[.6ex]
\centerline{ 
$\g^\sigma=\g_0=\bigoplus_{i\in \BZ}\g_0(2i)$, \ 
$\g^{\check\sigma}=\bigoplus_{i\in \BZ}\g(4i)$, and \ 
$\g^{\sigma\check\sigma}=\bigoplus_{i\in \BZ}\g_{\bar i}(2i)$\ ,}
\\[.6ex]
where $\bar i$ is the image of $i$ in $\BZ/2\BZ$.  Hence 
$\dim\g^\sigma-\dim\g^{\sigma\check\sigma}=\sum_{k\in\BZ} (d_0(4k+2)-d_1(4k+2))$, which proves the first assertion.
\\ \indent
If $\g_0$ is semisimple and $e\in\g_{0,\sf reg}$, then $d_0(0)=d_0(2)=\rk(\g_0)$. Hence 
$d_0(0)=d_1(2)$. It then follows from Corollary~\ref{cor:IBN} 
that $\mathsf{Sat}(\sigma)$ has {\sl IBN}. Since $\dim\g^\sigma=\dim\g^{\sigma\check\sigma}$ and both involutions have only {\sl IBN} in their Satake diagrams, they must be $G$-conjugate.
\end{proof}

\begin{rmk}
Actually, the semisimplicity of $\g_0$ is not necessary, as long as we 
know that $\mathsf{Sat}(\sigma)$ has only {\sl IBN}. The equalities $d_0(4k+2)=d_1(4k+2)$ are often
related to the fact that $G{\cdot}e$ is divisible and $G{\cdot}\edva\cap\g_1(4)\ne\varnothing$, cf. the
proof of Theorem~\ref{prop:dostat-del}. Hence Theorem~\ref{thm:conjugate} applies if
$d_0(0)=d_1(4)$. But there are many other cases, where $\sigma\sim\sigma\check{\sigma}$, see
e.g. Example~\ref{ex:Dn-hermit} with $n$ odd. 
\end{rmk}

\begin{ex}
It can happen that $\sigma=\vartheta_{\sf max}$ with $\g^\sigma$ semisimple, but the involutions 
$\sigma$ and $\sigma\check\sigma$ are not $G$-conjugate. Suppose that $\g=\GR{B}{4n+2}$ and 
$\sigma$ is of maximal rank, which is also a {\sl PI\/}--involution here. Then
$\g_0=\g^\sigma\simeq\g^{\check\sigma}\simeq \GR{B}{2n+1}\dotplus\GR{D}{2n+1}$, but 
$\g^{\sigma\check\sigma}=\GR{B}{2n}\dotplus\GR{D}{2n+2}$. The reason is that $G{\cdot}e$ is not 
divisible for $e\in\g_{0,\sf reg}$ and the equality $d_0(4k+2)=d_1(4k+2)$ fails exactly for 
$k=n, -n-1$. This leads to the relation $\dim\g^{\sigma\check\sigma}-\dim\g^\sigma=2$.
\end{ex}

\begin{ex}  \label{ex:three-non-conj}
There is an interesting case in which three involutions $\sigma, \check\sigma$, and 
$ \sigma\check\sigma$ are pairwise non-conjugate. Let 
$\sigma\in\mathsf{Inv}(\g)$ be a {\it diagram involution}\/ for $\g=\GR{A}{2n-1}$, $\GR{D}{n}$, and 
$\GR{E}{6}$. (See~\cite[Ch.\,8]{kac} for diagram automorphisms of simple Lie algebras.) Then 
$\sigma$ is outer, $\g^\sigma$ is simple, and if $e\in\g_{0,\sf reg}$, then $e\in\g_{\sf reg}$. 
By the above ``first properties'',  $\check\sigma$ is inner and $\sigma\check{\sigma}$ is outer.
Since $e\in\g_1^{(\check\sigma)}$, this implies that $\check\sigma=\vartheta$ is a {\sl PI}--involution.
Furthermore, $\sigma\check{\sigma}$ is the unique, up to $G$-conjugacy, outer involution in the
connected component of $\mathsf{Aut}(\g)$ containing $\sigma$ that has the property that
$e\in\g_1^{(\sigma\check\sigma)}$, i.e., the $(-1)$-eigenspace of the outer involution
$\sigma\check\sigma$ contains a regular nilpotent element of $\g$. If $\vartheta$ is not of maximal rank, 
then $\sigma\check\sigma$ appears to be of maximal rank. Below is the table with fixed point 
subalgebras for this triple of involutions, where $\sigma\check\sigma\not\sim \vartheta_{\sf max}$ only 
for $\g=\GR{D}{2n}$.

\begin{table}[ht]
\caption{Diagram involutions $\sigma$ and related triples}   \label{table:tri}
\begin{center}
\begin{tabular}{>{$}l<{$}| >{$}l<{$} >{$}l<{$} >{$}l<{$}|}
\g & \g^{\sigma} & \g^{\check\sigma}=\g^\vartheta & \g^{\sigma\check\sigma} \\ \hline\hline
\GR{A}{2n-1} & \GR{C}{n} & \GR{A}{n-1}\dotplus\GR{A}{n-1}\dotplus\te_1 & \GR{D}{n} \\
\GR{D}{2n} & \GR{B}{2n-1} & \GR{D}{n}\dotplus\GR{D}{n} & \GR{B}{n}\dotplus\GR{B}{n-1} \\
\GR{D}{2n+1} & \GR{B}{2n} & \GR{D}{n}\dotplus\GR{D}{n+1} & \GR{B}{n}\dotplus\GR{B}{n} \\
\GR{E}{6} &  \GR{F}{4} & \GR{A}{5}\dotplus\GR{A}{1} &  \GR{C}{4}  \\
\hline
\end{tabular}
\end{center}
\end{table}

For $\g=\GR{A}{2n}$, the diagram involution $\sigma$ is of maximal rank. Here this procedure provides
$G$-conjugate involutions $\sigma$ and $\sigma\check\sigma$.
\end{ex}

\end{document}